\documentclass[11pt]{amsart}

\usepackage{amssymb}
\usepackage{enumitem}
\usepackage[dvipsnames]{xcolor}
\usepackage[normalem]{ulem}

\usepackage{color}

\usepackage[a4paper,top=3 cm,bottom=3 cm,left=2.5 cm,right=2.5 cm]{geometry}

\newcommand{\cQ}{\mathcal{Q}}
\newcommand{\cP}{\mathcal{P}}

\newtheorem{thm}{Theorem}[section]

 \newtheorem{lem}{Lemma}[section]
 \newtheorem{prop}{Proposition}[section]
 \newtheorem{defn}{Definition}[section]
\newtheorem{rem}{Remark}[section]

\numberwithin{equation}{section}

\def\to{\rightarrow}
\newcommand{\Id}{\mathbb{I}}

\DeclareMathOperator{\dive}{div}

\newcommand{\M}{\mathcal{M}}

\renewcommand{\div}{\mathrm{div} \hspace{0.5mm}}     

\newcommand{\nchi}{{\raise.3ex\hbox{$\chi$}}}

\def\XXint#1#2#3{{\setbox0=\hbox{$#1{#2#3}{\int}$ }
		\vcen{\hbox{$#2#3$ }}\kern-.6\wd0}}

\newcommand{\N}{\mathbb{N}}

\newcommand{\R}{\mathbb{R}}

\newcommand{\Z}{\mathbb{Z}}

\makeatletter
\newcommand{\justified}{%
	\rightskip\z@skip%
	\leftskip\z@skip}
\makeatother

\newcommand\restr[2]{{
		\left.\kern-\nulldelimiterspace 
		#1 
		\vphantom{\big|} 
		\right|_{#2} 
}}

\renewcommand{\d}{{\rm d}}

\newcommand{\var}{\varepsilon}

\newcommand{\la}{\lambda}

\DeclareFontFamily{U}{mathx}{\hyphenchar\font45}
\DeclareFontShape{U}{mathx}{m}{n}{<-> mathx10}{}
\DeclareSymbolFont{mathx}{U}{mathx}{m}{n}
\DeclareMathAccent{\widebar}{0}{mathx}{"73}

\newcommand{\lb}{\langle}

\def\d{\partial}

\renewcommand{\i}{\ifmmode\mathit{\mathchar"7010 }\else\char"10 \fi}
\renewcommand{\j}{\ifmmode\mathit{\mathchar"7011 }\else\char"11 \fi}

\renewcommand{\le}{\leq} 
\renewcommand{\ge}{\geq}

\def\char{{1\!\mbox{\rm l}}}

\allowdisplaybreaks

\usepackage[utf8]{inputenc}  
\usepackage[T1]{fontenc}
\usepackage{lmodern}         
\input{glyphtounicode}

\usepackage[colorlinks,
linkcolor=red,
citecolor=blue
]{hyperref}
\usepackage{color} 

\def\Id{{\rm Id}\,}
\def\muu{{\bar\mu}}
\def\d{\partial}
\def\ddj{\dot \Delta_j}

\def\tilde{\widetilde}
\def\hat{\widehat}

\def \zr1{$z_{R,1}$}
\def \zr2{z_{R,2}}
\def \zi1{z_{I,1}}
\def \zi2{z_{I,2}}

\def\cP{{\mathcal P}}
\def\cQ{{\mathcal Q}}

\def\la{{\lambda_+}}
\def\lb{{\lambda_-}}

\begin{document}

\title[Electrostatic effects on viscous compressible fluids]{Electrostatic effects on critical regularity and long-time behavior of viscous compressible fluids}

\author{Ling-Yun Shou and Zihao Song}

\date{}




\begin{abstract}
We consider the compressible Navier–Stokes–Poisson equations in $\mathbb{R}^d$ ($d\geq2$), a classical model for barotropic compressible flows coupled with a
self-consistent electrostatic potential. We show that the electrostatic coupling has a significant impact on the long-time dynamics of solutions due to its underlying Klein–Gordon structure. As a first result, we prove the global well-posedness of the Cauchy problem with initial data near equilibrium in the full-frequency $L^{p}$-type critical Besov space \emph{without relying on hyperbolic symmetrization}. Compared with the Poisson-free case  studied in several milestone works [\text{Charve and Danchin}, Arch. Rational Mech. Anal., 198 (2010), pp.~233--271;\
\text{Chen, Miao and Zhang}, Commun. Pure Appl. Math., 63 (2010), pp.~1173--1224;\
\text{Haspot}, Arch. Rational Mech. Anal., 202 (2011), pp.~427--460], we  remove the extra $L^{2}$ assumption in low frequencies and extend the
admissible choice of $p$ to the {\emph{sharp}} range $1\leq p<2d$. This is, to the best of our knowledge, the first result in compressible fluids that allows the initial velocity field to be highly oscillatory across all frequencies.

Furthermore, stemming from the Poisson coupling, the density and velocity exhibit distinct low-frequency behaviors. Motivated by this feature, we propose a general $L^{p}$-type low-frequency assumption and 
establish the optimal convergence rates of global solutions toward equilibrium. For a broad class of indices, this assumption yields faster decay than those
obtained under the classical $L^{1}$ framework. To this end, we develop a time-weighted energy method, which is of interest and enables us to capture maximal decay estimates without additional smallness of initial data.


\end{abstract}

\maketitle

\section{Introduction}\setcounter{equation}{0}
Electrostatic interactions constitute an essential component of compressible
fluid models across a broad range of physical settings, including plasma physics,
semiconductor modeling, and charged particle flows \cite{3,21}. When such interactions are generated by the density distribution, they induce a self-consistent feedback mechanism between the fluid motion and the electric field, resulting in
nonlocal effects that significantly influence both
well-posedness and qualitative behaviors of solutions.

In this paper, we investigate the following hydrodynamic system with electrostatic interactions in $d\geq2$:
\begin{equation}
\left\{
\begin{aligned}
&\partial_{t}\rho+\mathrm{div}(\rho u)=0,\\ 
&\partial_{t}(\rho u)+\mathrm{div}(\rho u \otimes u)+\nabla P(\rho)
=\mathcal{A}(u)+\kappa \rho\nabla \phi,\\
&\Delta \phi=\rho-\rho^*,
\end{aligned}
\right.
\label{1.1}
\end{equation}
which, referred to as the compressible Navier-Stokes-Poisson (CNSP) system,
consists of the compressible Navier-Stokes equations coupled with a Poisson equation for Coulomb-type interactions. Here,
$\rho=\rho(t,x)\in \mathbb{R}_{+}$ and $u=u(t,x)\in \mathbb{R}^{d}$ denote
the fluid density and velocity field on $[0,+\infty)\times \mathbb{R}^{d}$,
respectively, and $\phi=\phi(t,x)$ is the associated self-consistent electrostatic
potential. We assume that the fluid is barotropic, so that the pressure depends only on the
density, namely $P=P(\rho)$. The viscous operator $\mathcal{A}$ is defined by
\[
\mathcal{A}(u)=\mathrm{div}\bigl(2\mu_1(\rho)D(u)\bigr)
+\nabla\bigl(\mu_2(\rho)\mathrm{div}u\bigr),
\]
where $\mu_1(\rho)$ and $\mu_2(\rho)$ denote the density-dependent shear and bulk
viscosity coefficients, respectively,  $D(u)=\tfrac12(\nabla u+\nabla^T u)$ stands for the deformation tensor, and $\mathrm{div}$ is
the divergence operator with respect to the space variable. To ensure the uniform ellipticity of the viscous operator $\mathcal A$ and the hyperbolicity of the system linearized around the equilibrium density $\rho^*$,
we assume that
\[
\mu_1(\rho^*)>0,
\quad
\mu_1(\rho^*)+2\mu_2(\rho^*)>0\quad \text{and}\quad P'(\rho^*)>0\quad \text{for}\quad \rho^*>0.
\]
We consider the Cauchy problem for system \eqref{1.1} with the initial condition
\begin{equation}
(\rho,u)|_{t=0}=(\rho_0(x),u_0(x)), \qquad x\in\mathbb{R}^d,
\label{1.2}
\end{equation}
and assume that the initial data are perturbations of the constant equilibrium
state, i.e. in the sense that
$(\rho_0(x),u_0(x))\to(\rho^*,0)\quad \text{as } |x|\rightarrow\infty$
where $\rho^*>0$ is a given constant.

The parameter $\kappa$ characterizes the nature of interactions. 
The case $\kappa>0$ corresponds to repulsion and arises in models of charged particles under a self-consistent electrostatic field~\cite{21}, whereas $\kappa<0$ corresponds to attraction and is related to self-gravitating gaseous stars~\cite{3}. 
Linearizing around $(\rho,u)=(\rho^*,0)$ and working in Fourier variables, one finds that the equilibrium is linearly stable for $\kappa>0$ but unstable for $\kappa<0$. 
Consequently, we restrict ourselves to the repulsive regime $\kappa>0$.

\subsection{Compressible Navier-Stokes equations}

The evolution of compressible
viscous flows obeys the following barotropic compressible Navier-Stokes (CNS) system
\begin{equation}
\left\{
\begin{aligned}
&\partial_{t}\rho+\div(\rho u)=0,\\
&\partial_{t}(\rho u)+\div(\rho u \otimes u)+\nabla P(\rho)=\mathcal{A}(u).
\end{aligned}
\right.
\label{NS}
\end{equation}
Without the Poisson coupling, that is, when $\kappa=0$, the system \eqref{1.1} reduces to  \eqref{NS}. 

The CNS system \eqref{NS} has been studied extensively, and many significant results have been obtained. The local existence and uniqueness of smooth solutions were established by Serrin~\cite{serrin1} and Nash~\cite{nash1}. For weak solutions with finite-energy initial data, a major breakthrough is due to Lions~\cite{lions1}. Further developments were later achieved by Feireisl, Novotn\'y and Petzeltov\'a~\cite{feireisl1}, as well as by Jiang and Zhang~\cite{jiang1}, and the subject has remained very active since then. However, the uniqueness of solutions for general large data remains largely open. For small initial perturbations of a constant equilibrium in $H^3(\mathbb{R}^3)$, Matsumura and Nishida~\cite{mats1} established the first global well-posedness result for the compressible and heat-conductive Navier--Stokes equations. Furthermore, assuming in addition that the initial data belong to $L^1(\mathbb{R}^3)$, they~\cite{mats2} obtained heat-like decay estimates in $L^2(\mathbb{R}^3)$. Later, Ponce~\cite{ponce1} derived decay estimates in general Sobolev norms.  Hoff and Zumbrun \cite{hoff2} investigated the $L^p(\mathbb{R}^3)$ decay estimates for the isentropic Navier–Stokes system and proved the asymptotic stability of diffusion waves. Later, Liu and Wang \cite{liuwang} derived the pointwise estimates of solutions in odd dimensions and showed the generalized Huygens principle. For weaker initial data close to no-vacuum equilibrium states with discontinuous density, Hoff~\cite{hoff1,hoff97} obtained global intermediate solutions between strong and weak solutions. The presence of vacuum is one of the main difficulties in well-posedness theory. Indeed, Xin~\cite{xin0} showed that any smooth solution may blow up in inhomogeneous Sobolev spaces in finite time if the initial density contains vacuum. On the other hand, Huang, Li and Xin~\cite{HLX} constructed global classical solutions in inhomogeneous Sobolev spaces allowing large oscillations and vacuum states.


It is therefore natural to seek a functional framework with the lowest possible regularity in which the well-posedness theory still holds, that is, a space that is as close as possible to the basic energy level from a mathematical viewpoint. 
Based on the scaling invariant (originating from Fujita and
Kato \cite{fujita1} for the incompressible case), a classical choice of critical space for the CNS system \eqref{NS} is $\dot{B}^{d/p}_{p,1}\times \dot{B}^{d/p-1}_{p,1}$. 
Danchin~\cite{danchin1} first established global well-posedness for small initial data in the $L^2$-type critical homogeneous Besov space. 
We also refer to~\cite{danchin3} for local well-posedness with general large initial data in $L^p$-based critical regularity for $1<p<2d$. 
The global existence result was later extended by Charve and Danchin~\cite{charve1}, and independently by Chen, Miao and Zhang~\cite{chen1}, to general $L^2$-$L^p$-based Besov spaces.
Inspired by Hoff’s viscous effective flux~\cite{hoff1}, Haspot~\cite{haspot1} developed an  energy method in high frequencies and obtained essentially the same class of results. To overcome the hyperbolic structure (acoustic wave) in low frequencies, to the best of our knowledge, all global existence results restricted the index $p$ satisfying 
\begin{align}\label{pold}
p\in\bigl[2,\min\{4,\tfrac{2d}{d-2}\}\bigr]\quad\text{and}\quad p\neq4 \quad \text{for}\quad  d=2,
\end{align}
and an additional $L^2$ low-frequency assumption was imposed:
\begin{align}\label{addition}
(\rho_0-\rho^*, u_0)^\ell\in\dot{B}^{\frac{d}{2}-1}_{2,1}. 
\end{align}

The large-time asymptotic behavior of global solutions in critical spaces has received considerable attention. 
Okita~\cite{okita1} established decay estimates in the critical $L^2$ framework under an additional low-frequency assumption in $\dot{B}^{0}_{1,\infty}$ for $d\geq3$. Danchin proposed an alternative description of time decay that applies to all spatial dimensions $d\ge2$.
Subsequently, Danchin and Xu~\cite{danchin5} obtained decay rates in $L^2$-$L^p$-type critical spaces under the additional assumption that the low-frequency $\dot{B}^{\sigma_1}_{2,\infty}$-norm ( $d/2-2d/p\leq \sigma_1<d/2-1$) of $(\rho_0-\rho^*, u_0)$ is sufficiently small. Here $d/2-2d/p$ corresponds to the critical embedding
$L^{p/2}\hookrightarrow\dot{B}^{\sigma_0}_{2,\infty}$ for $2\le p\le\min\{4,2d/(d-2)\}$.
These results rely essentially on time-weighted energy methods in the Fourier semigroup framework, where the smallness of the low-frequency component of the initial data typically plays a crucial role.
Later, the smallness requirement in $\dot{B}^{\sigma_1}_{2,\infty}$ was removed by Xin and Xu~\cite{xin1}  using a Lyapunov-type energy method in the spirit of Guo and Wang~\cite{guo1}. Inspired by the decay characterization and Wiegner's theorem for viscous incompressible fluids, Brandolese, the first author, Xu and Zhang~\cite{BSXZ} further refined the Fourier semigroup approach to derive decay estimates in critical spaces and proved that the low-frequency $\dot{B}^{\sigma_1}_{2,\infty}$-regularity of the initial perturbation is both sufficient and necessary for upper bounds of decay rates.
They also characterized both lower and upper bounds of decay rates in critical spaces in terms of the subset $\mathcal{B}^{\sigma}_{2,\infty}$ of the low-frequency part of the initial perturbation.

\subsection{Compressible Navier-Stokes-Poisson equations}

Compared with the CNS equations \eqref{NS}, an additional lower-order Poisson coupling in the CNSP system \eqref{1.1} leads to different phenomena and extra mathematical difficulties. The theory of weak solutions was investigated by Donatelli~\cite{donatelli1},  Zhang and Tan~\cite{zhang1} and Duan and Li \cite{duanli1}. Concerning strong solutions  near a constant equilibrium, Li, Matsumura and Zhang~\cite{LiNSP0} established global well-posedness in three-dimensional Sobolev spaces by adapting the classical Matsumura and Nishida method~\cite{mats1,mats2}. Under the additional condition that both $\rho_0-\rho^*$ and $u_0$ are small in $L^1(\mathbb{R}^3)$, they \cite{LiNSP0} also obtained optimal rates of global solutions:
\begin{align*}
\|(\rho-\rho^*)(t)\|_{L^2}\lesssim (1+t)^{-\frac{3}{4}},\quad \|u(t)\|_{L^2}\leq C(1+t)^{-\frac{1}{4}},
\end{align*}
where the $L^2$-rate of the velocity is slower than that of the classical CNS system due to the influence of the electrostatic force. Wang and Wu \cite{wang1} investigated the pointwise estimates of solutions and showed that the pointwise profile of the solution contains the Diffusion wave but does not contain the Huygens wave, which is different from the case of the CNS system. Li and Zhang \cite{LiNSP1} established decay estimates of classical solutions in the case where the initial perturbation is additionally small in $\dot{B}^{-s}_{1,\infty}$ with some $s\geq0$.   Using the method developed by Guo and Wang \cite{guo1}, Wang \cite{wang2} enhanced the decay rate by considering the coupling relationship of $\rho-\rho^*$, $m$ and under the boundedness of the negative Sobolev norm $\dot H^s$ ($s\in[0,3/2)$). We also mention the stability of solutions around certain nontrivial steady states studied in a series of works by Luo, Xin and Zeng~\cite{luoxinzeng1,luoxinzeng2}.

On the other hand, it is interesting to consider the global dynamics of  solutions in the critical regularity setting. More precisely, the solutions are sought in a functional space that has the same invariance under time and space dilations as the CNSP system \eqref{1.1} itself (neglecting the lower-order pressure and electrostatic field terms), namely 
$$
(\psi,\rho,u)\longrightarrow (\psi_l,\rho_l,u_l)\quad\text{ for all }\,\,\,  l>0
$$
with
\begin{equation*}
\psi_l(t,x)=l^{-2}\psi(l^{2}t,l x),\quad
\rho_l(t,x)=\rho(l^{2}t,l x),\quad
u_l(t,x)=l\,u(l^{2}t,l x).
\end{equation*}
 In addition, the Stokes maximal-regularity estimate for the second equation in \eqref{1.1} forces the Poisson term to lie in the same space as the velocity. Consequently, a suitable choice of data may be $(\dot{B}^{d/p-2}_{p,1}\cap \dot{B}^{d/p}_{p,1})\cap \dot{B}^{d/p-1}_{p,1} $. Hao and Li \cite{hao} established the 
global existence of small strong solutions to \eqref{1.1}-\eqref{1.2} in the $L^2$ hybrid Besov space and in dimensions $d\geq3$. To extend the  existence theory to $d\geq2$, Chikami and Ogawa \cite{CO} established local well-posedness when the initial data belong to the hybrid Besov space $\dot{B}^{d/p-2+\nu,d/p}_{p,1}\times\dot{B}^{d/p-1+\nu,d/p-1}_{p,1}$ with $1<p<2d/(2-\nu)$, $0<\nu\leq1$ if $n=2$ and $\nu=0$ if $\nu\geq3$. 
Subsequently, Chikami and Danchin \cite{CD} proved the unique global strong 
solution within the $L^p$ Besov space $(\dot{B}^{d-2}_{p,1}\cap\dot{B}^{d-2}_{p,1})\cap \dot{B}^{d-1}_{p,1} $ for $d\geq2$ under the same restriction \eqref{pold} and the additional low-frequency $L^2$ assumption
\begin{align}\label{L2NSP}
\|\rho_0-\rho^*\|_{\dot{B}^{\frac{d}{2}-2}_{2,1}}^{\ell}+\| u_0\|_{\dot{B}^{\frac{d}{2}-1}_{2,1}}^{\ell}\ll 1.
\end{align}
Moreover, they \cite{CD} prove optimal time-decay estimates of solutions and higher-order derivatives in the $L^2$ setting:
\begin{align*}
\|(\rho-\rho^*)(t)\|_{\dot{B}^{\sigma-1}_{2,1}}+\|u(t)\|_{\dot{B}^{\sigma}_{2,1}}\leq (1+t)^{-\frac{\sigma}{2}-\frac{d}{4}},\quad -\frac{d}{2}<\sigma\leq \frac{d}{2}+1
\end{align*}
if additionally the low frequencies of $\rho_0-\rho^*$ and $u_0$ are small in $\dot{B}^{-d/2-1}_{2,\infty}\times \dot{B}^{-d/2}_{2,\infty}$. The decay result of \cite{CD} has been extended to the $L^2$-$L^p$ setting.  By additionally considering the $\dot{B}^{\sigma_1}_{2,\infty}$-type smallness assumption $\|\rho_0-\rho^*\|_{\dot{B}^{\sigma_1-1}_{2,\infty}}^{\ell}+\|u_0\|_{\dot{B}^{\sigma_1}_{2,\infty}}^{\ell}\ll 1$  with $d/2-2d/p\leq \sigma_1<d/2-1$,  Shi and Xu \cite{shi} established the large time behavior
\begin{align*}
\|(\rho-\rho^*)(t)\|_{\dot{B}^{\sigma-1}_{2,1}}^{\ell}+\|u(t)\|_{\dot{B}^{\sigma}_{2,1}}^{\ell}+\|(\rho-\rho^*,\nabla u)(t)\|_{\dot{B}^{\frac{d}{p}}_{p,1}}^{h}\lesssim (1+t)^{-\frac{1}{2}(\sigma-\sigma_1)},\quad\sigma_1<\sigma\leq \frac{d}{2}+1.
\end{align*}

\vspace{2mm}

It is well known that the $L^{2}$-based low-frequency assumptions (e.g., \eqref{addition} and \eqref{L2NSP}) play a crucial role in the global theory for compressible fluids, as pointed out in, for instance, in \cite{CD,charve1,haspot1}, owing to the symmetric hyperbolic structure. 
By contrast, for local-in-time solutions,  both low and high frequencies may be treated in $L^p$-type critical spaces with the same index $p\in[1,2d)$. Therefore, a natural open question is whether this gap can be bridged.

The purpose of this work is to provide a positive answer to this question for compressible viscous fluids with electrostatic interactions. 
We prove that the electrostatic coupling indeed produces a low-frequency regularizing effect. As a consequence, we establish global well-posedness for the CNSP system \eqref{1.1} for initial data near equilibrium in the critical space $(\dot{B}^{d/p-2}_{p,1}\cap \dot{B}^{d/p}_{p,1})\cap \dot{B}^{d/p-1}_{p,1}$  with  the {\emph{optimal}} range $p\in [1,2d)$, which is \emph{independent of hyperbolic symmetrization}. Furthermore, we introduce a new $L^p$-type low-frequency boundedness condition and quantitatively investigate the long-term behaviors with different decay properties for the density and the velocity. 

\subsection{Main results}

We state our first main result on the global existence and uniqueness of solutions to the Cauchy problem \eqref{1.1}-\eqref{1.2}.

\begin{thm}\label{thm2.1}
Let $1\leq p<2d$. If there exists a positive $\varepsilon_{0}>0$ such that the initial datum $(\rho_{0},u_0)$ fulfills $ \rho_{0}-\rho^*\in \dot{B}^{\frac{d}{p}-2}_{p,1}\cap \dot{B}^{\frac{d}{p}}_{p,1}$, $u_{0}\in\dot{B}^{\frac{d}{p}-1}_{p,1}$ and
\begin{eqnarray}
\|\rho_{0}-\rho^*\|_{\dot{B}^{\frac{d}{p}-2}_{p,1}\cap \dot{B}^{\frac{d}{p}}_{p,1}}+
\|u_{0}\|_{\dot{B}^{\frac{d}{p}-1}_{p,1}}\leq \varepsilon_{0},\label{a1}
\end{eqnarray}
then the Cauchy problem \eqref{1.1}-\eqref{1.2} admits a unique global-in-time solution $(\rho,u)$ satisfying 
\begin{align}\label{1.8}
&\rho-\rho^*\in\mathcal{C}(\mathbb{R}_+;\dot{B}^{\frac{d}{p}-2}_{p,1}\cap \dot{B}^{\frac{d}{p}}_{p,1})\cap L^1(\mathbb{R}_+;\dot{B}^{\frac{d}{p}}_{p,1}), \quad u\in\mathcal{C}(\mathbb{R}_+;\dot{B}^{\frac{d}{p}-1}_{p,1})\cap L^1(\mathbb{R}_+;\dot{B}^{\frac{d}{p}+1}_{p,1}),
\end{align}
and
\begin{equation}\label{bound}
\begin{aligned}
\|\rho-\rho^*\|_{\widetilde{L}^\infty(\mathbb{R}_+;\dot{B}^{\frac{d}{p}-2}_{p,1}\cap\dot{B}^{\frac{d}{p}}_{p,1})\cap L^1(\mathbb{R}_+;\dot{B}^{\frac{d}{p}}_{p,1})}&+\|u\|_{\widetilde{L}^\infty(\mathbb{R}_+;\dot{B}^{\frac{d}{p}-1}_{p,1})\cap L^1(\mathbb{R}_+;\dot{B}^{\frac{d}{p}+1}_{p,1})}\\
&\leq C\big(\|\rho_{0}-\rho^*\|_{\dot{B}^{\frac{d}{p}-2}_{p,1}\cap\dot{B}^{\frac{d}{p}}_{p,1}}+
C\|u_{0}\|_{\dot{B}^{\frac{d}{p}-1}_{p,1}}\big),
\end{aligned}
\end{equation}
with $C>0$ a generic constant.
\end{thm}

\begin{rem}\normalfont
There are several comments concerning Theorem~\ref{thm2.1}.

\begin{itemize}

\item 
Theorem~\ref{thm2.1} removes the additional low-frequency $L^2$ assumptions \eqref{L2NSP} and extends the usual admissible range $
p\in\bigl[2,\min\{4,\tfrac{2d}{d-2}\}\bigr]$ ($p\neq4$ for $d=2$)  required in the seminal results  ~\cite{charve1,chen1,haspot1} for the CNS system and \cite{CD} for the CNSP system, which allows constructing global solutions with a larger class of initial data.
The key ingredient of our analysis is to capture the coupled diffusive–Klein–Gordon structure at low frequencies, which yields an {\emph{enhanced regularization effect}} compared with that of the acoustic wave (see \ref{1.23}).

\item 
When $p>d$, global existence and uniqueness hold for the velocity in critical spaces with negative regularity.
In particular, the velocity may be highly oscillatory across {\emph{all frequencies}}. A classical example is $u_0(x)=\sin (\frac{x_1}{\var})\phi(x)$ with any $\phi(x)\in \mathcal{S}(\mathbb{R}^d)$, which satisfies $$
\|u_0\|_{\dot{B}^{\frac{d}{p}-1}_{p,1}}\lesssim \var^{1-\frac{d}{p}}\ll1
$$
when $\var$ is sufficiently small. This is an analogy to Cannone's theorem \cite{cannone1} for the incompressible Navier-Stokes equations.

\item 
The restriction $1\leq p<2d$ appears to be {\emph{sharp}}.
Indeed, strong ill-posedness phenomena occur when $p\geq 2d$; see \cite{chen2,IQ1}.
Such a full range of $p$, including the endpoint Lebesgue exponent, is consistent with the local well-posedness theory (see~\cite{CL} and Theorem~\ref{thm3.1}).
The critical case $p=1$ is excluded in Theorem~\ref{thm2.1} thanks to a priori estimates remaining valid under small initial perturbations.

\end{itemize}
\end{rem}

Then, we establish the large-time behavior of solutions under a new $L^p$ assumption. 

\begin{thm}\label{Theorem2.2}
Let $(\rho,u)$ be the global solution to the Cauchy problem \eqref{1.1}-\eqref{1.2} constructed in Theorem \ref{thm2.1}. Define
the momentum $m=\rho u$ and its initial datum $m_0=\rho_0 u_0$. Let the real number $\sigma_{1}$ satisfy \vspace{-2mm} 
\begin{align}
   \sigma_0-1\leq \sigma_{1}<\frac{d}{p}-1\quad\text{with} \quad \sigma_0 =
\begin{cases}
-\dfrac{d}{p}, & 2\leq p< 2d,\\[6pt]
-\dfrac{d}{p'}, & 1\le p<2.
\end{cases}\label{sigma1}
\end{align}
Here $p'$ is given by $\frac{1}{p}+\frac{1}{p'}=1$. If the initial datum $(\rho_0,u_0)$ satisfies
\begin{align}
    (\Lambda^{-1}(\rho_0-\rho^*),\rho_{0} u_0)^{\ell} \in \dot{B}^{\sigma_{1}}_{p,\infty}, \label{aa1}
\end{align}
then for all $t\geq 1$, the solution $(\rho,u)$ admits
\begin{align}
\|(\rho-\rho^*)(t)\|_{\dot{B}^{\sigma}_{p,1}}&\leq C(1+ t)^{-\frac{1}{2}(\sigma-\sigma_{1}+1)},\,\,\,\,\,\,\quad\quad\quad \quad\sigma_1-1<\sigma\le \dfrac{d}{p},\label{decay1}\\
\|\rho u(t)\|_{\dot{B}^{\sigma}_{p,1}}&\leq C(1+ t)^{-\frac{1}{2}(\sigma-\sigma_{1})},\quad\quad\,\,\,\,\quad\quad\quad\quad\quad \sigma_1<\sigma\le \dfrac{d}{p},\label{decay4}\\
\|u(t)\|_{\dot{B}^{\sigma}_{p,1}}&\leq C(1+ t)^{-\frac{1}{2}(\sigma-\sigma_{1})},\quad\quad\quad \,\,\,\, \min\!\left\{\sigma_0,\sigma_1\right\}<\sigma\le \dfrac{d}{p}+1,\label{decay3}
\end{align}
with $C>0$ a uniform constant independent of $t$.
\end{thm}

\begin{rem}\normalfont
Some comments concerning Theorems \ref{Theorem2.2} are in order.
\begin{itemize}



\item We highlight a qualitative impact of electrostatic interactions on compressible viscous flows. On the one hand, the eigenvalues of the linearized CNSP system display mixed  heat diffusion and Klein–Gordon–type oscillations, which allows us to introduce, for the first time, $L^{p}$-type ($p\neq 2$) decay assumption \eqref{aa1} in the compressible fluid setting. Such a generalization naturally allows us to establish uniform decay within a larger class of initial data by noticing the following embedding chain:
$$L^1 \hookrightarrow \dot B^{0}_{1,\infty}\hookrightarrow \dot B^{-\frac{d}{2}}_{2,\infty}\hookrightarrow \dot B^{-\frac{d}{p}}_{p,\infty},\quad p\geq2.$$
On the other hand, the Poisson coupling alters the eigenprojections such that the density and the velocity (or the momentum) lie at different regularities, leading to faster decay of the density, in sharp contrast with the barotropic compressible Navier-Stokes equations \cite{BSXZ,danchin5,xin1}.

\item Theorem \ref{Theorem2.2} depends on a new weighted energy method relying only on the boundedness of solutions and may be of independent interest. In particular, this allows us to remove the additional smallness condition in earlier significant works \cite{CD,LiNSP0,LiNSP1,shi}.

\item When $p=2$, our low-frequency  assumption \eqref{aa1} reduces to $(\Lambda^{-1}(\rho_0-\rho^*),\,\rho_0u_0)^{\ell}\in \dot{B}^{\sigma_1}_{2,\infty}$ ($-\frac{d}{2}-1\leq \sigma_1<\frac{d}{2}-1$) used in Chikami and Danchin \cite{CD} and Shi and Xu \cite{shi}.
For $-\frac{d}{2}\leq\sigma_1\le \frac{d}{2}-1$, the conditions of $\rho_0u_0\in \dot{B}^{\sigma_1}_{2,\infty}$ and $u_0\in \dot{B}^{\sigma_1}_{2,\infty}$ are equivalent thanks to product laws. 
Moreover, for $-\frac{d}{2}-1\leq \sigma_1<-\frac{d}{2}$, our approach yields decay rates that are even faster than those obtained under the classical $L^{1}$ assumption.

\item When $p=1$, the low-frequency assumption becomes $(\Lambda^{-1}(\rho_0-\rho^*),\,\rho_0u_0)^{\ell}\in \dot{B}^{\sigma_1}_{1,\infty}$ ($-1\le \sigma_1<d-1$), which essentially coincides with the assumption introduced by Li and Zhang \cite{LiNSP1}. 
Specifically, if $\sigma_1=-1$, we are able to obtain the asymptotic behavior in $L^1$:
\begin{align*}
\|(\rho-\rho^*)(t)\|_{L^1}&\leq C(1+t)^{-\frac{3}{2}},\quad \|u(t)\|_{L^1}\leq C(1+t)^{-\frac{1}{2}}.
\end{align*}
These imply the uniform decay of $(\rho-\rho^*,u)$ in integrable spaces rather than the classical boundedness.

\end{itemize}
\end{rem}


\subsection{Reformulation and methodology}

Without loss of generality, we shall fix the equilibrium of the density to be $\rho^{\ast}=1$. 
Denoting the density fluctuation  $a=\rho-1$, we reformulate the Cauchy problem \eqref{1.1}-\eqref{1.2} as
\begin{equation}
\left\{
\begin{array}{l}\partial_{t}a+u\cdot\nabla a+(1+a)\mathrm{div}u=0,\\ [1mm]
 \partial_{t}u-\bar{\mathcal{A}}u+\gamma\nabla a+\kappa\frac{\nabla}{-\Delta} a= g,\\[1mm]
(a,u)|_{t=0}=(a_{0},u_{0}).\\[1mm]
 \end{array} \right.\label{linearized:1}
\end{equation}
Here, $\gamma=P'(1)$, the linearized viscous operator $\bar{\mathcal{A}}$ is given by
$$\bar{\mathcal{A}}u=\bar\mu_1\Delta u+(\bar\mu_1+\bar\mu_2)\nabla\dive u\quad\text{with}\quad  \bar\mu_1=\mu_1(1)\quad \text{and}\quad \bar\mu_2=\mu_2(1).$$ One can write the nonlinear terms $g=g_1+g_2+g_3$ as follows: 
\begin{equation}
\begin{aligned}
g_1=-u\cdot\nabla u,\quad g_{2}=\big(1+Q(a)\big)\big(\mathrm{div}(2\tilde{\mu}_1(a) D(u))+\nabla(\tilde{\mu}_2(a)\mathrm{div}u)\big),\quad 
g_{3}={G}(a)\nabla a,
\end{aligned} \label{nonlinear}
\end{equation}
with
\begin{align*}
Q(a)&=\frac{1}{a+1}-1,\quad\quad\quad\,\, G(a)=(a+1)P'(a+1)-\lambda,\\
\tilde{\mu}_1(a)&=\mu_1(a+1)-\bar\mu_1,\quad\tilde{\mu}_2(a)=\mu_2(a+1)-\bar\mu_2.
\end{align*}

To better understand the proofs of Theorems \ref{thm2.1} and \ref{Theorem2.2}, we first perform a formal spectral analysis of the linearized system associated with \eqref{1.1}, whose precise formulation will be given in Section~3. Under the incompressible–compressible decomposition $u=\cP u+\cQ u$ where $\cP=\mathrm{Id}-\frac{\nabla\div}{-\Delta}$, it is not difficult to see that the divergence-free part $\cP u$ fulfills a pure heat equation. For the density $a$ and the compressible part of the velocity $\cQ u$, in light of scaling, we denote $\mathcal{U}=\Lambda^{-1}a$ and $\mathcal{V}=\frac{\div}{\Lambda}\cQ u$ so that the main ingredient is to analyze the following coupling system:
\begin{equation}\label{R-E69}
\left\{\begin{array}{l}\d_t\mathcal{U}+ \mathcal{V}=0,\\[1ex]
\d_t\mathcal{V}-\bar\mu\Delta \mathcal{V}-\gamma\Lambda^{2} \mathcal{U}-\kappa \mathcal{U}=0,
\end{array}\right.
\end{equation}
where we simply write  $\bar\mu=2\bar\mu_1+\bar\mu_2$.
Taking the Fourier transform of \eqref{R-E69} with respect to $x\in \mathbb{R}^{d}$ leads to
\begin{equation}
\frac{d}{dt}\left(\begin{array}{c}
\widehat{\mathcal{U}} \\
\widehat{\mathcal{V}} \\
\end{array}\right)
=A(\xi)\left(\begin{array}{c}
\widehat{\mathcal{U}} \\
\widehat{\mathcal{V}} \\
\end{array}\right)
\quad \mbox{with}\quad A(\xi)=\left(
\begin{array}{cc}0 & -1 \\
\kappa+\gamma|\xi|^2 & -\bar\mu|\xi|^2 \\
\end{array}\right).
\end{equation}
To illustrate how we remove the additional $L^2$ assumption at low frequencies, we examine the eigenvalues of $A(\xi)$ in the regime $|\xi|\ll1$:
\begin{eqnarray}\label{eigenvalue}
\lambda_{\pm}=-\frac{\bar\mu|\xi|^2}{2}\pm {\rm i} \sqrt{\kappa-\gamma|\xi|^2-\frac{\bar\mu^2}{4}|\xi|^4}.
\end{eqnarray}
Such eigenvalues naturally generate a diffusive–dispersive semigroup, which is denoted by
$e^{\frac{\bar\mu}{2}\Delta t+\M t}$ 
with $\M$ defined through the Fourier symbol $$\hat\M(\xi)=i\sqrt{\kappa-\gamma|\xi|^2-\frac{\bar\mu^2}{4}|\xi|^4}.$$
We now explain how the dispersive component $e^{\M t}$, under different regimes of the coefficient $\kappa$, interacts with the diffusive part $e^{\frac{\bar\mu}{2}\Delta t}$, which behaves like the heat kernel. In light of Peral's work \cite{peral1}(see also \cite{li1,guo2}), one could infer the following $L^p$ bounds:
\begin{equation}\label{1.23}
\|\ddj e^{\M t}f\|_{L^p}\lesssim\left\{
\begin{aligned}
&|t2^j|^{\frac{d-1}{2}|1-\frac{p}{2}|}\|\ddj f\|_{L^p},\quad \kappa=0,\\ 
&|t2^{2j}|^{\frac{d}{2}|1-\frac{p}{2}|}\|\ddj f\|_{L^p},\quad\,\,\, \kappa>0,
\end{aligned}
\right.\quad\quad \mathrm{for}\quad 2^j<<1.
\end{equation}
Consequently,
\begin{itemize}
\item If $\kappa=0$, the dispersion relation is of the so-called acoustic-wave type, and the corresponding diffusive-dispersive operator $e^{\frac{\bar\mu}{2}\Delta t+\M t}$ may exhibit singularities in $L^{p}$ ($p\neq2$) owing to the weak boundedness of the wave operator (see also Brenier \cite{brenner} regarding the ill-posedness results for hyperbolic equations).

\item If $\kappa>0$, the imaginary part behaves analogously to the linear operator of the Klein-Gordon equation, which shares a {\emph{"higher order" factor}} of $t 2^{2j}$ consistent with the scale of the heat kernel. This inspires us to capture a low-frequency regularization effect for $e^{\frac{\bar\mu}{2}\Delta t+\M t}$ and allows us to consider $L^p$ type estimates exactly the same as those for the heat equation, in sharp contrast with the $L^2$ framework for the purely compressible Navier-Stokes equations.
\end{itemize}
Based on the above analysis, if one lets $\{\mathcal{G}(t)\}_{t\ge0}$ denote the semigroup generated by the linear system \eqref{linearized:1}, then the solution operator associated with \eqref{R-E69} satisfies the following low-frequency estimates for $p\in[1,\infty]$:
\begin{equation}\label{new kernel}
\|\ddj \mathcal{G}(t)(\Lambda^{-1}a_0,u_0)\|_{L^p}
\lesssim
e^{-c2^{2j}t}
\|\ddj (\Lambda^{-1}a_0,u_0)\|_{L^p},
\qquad 2^j<<1,
\end{equation}
for some constant $c>0$. We refer to Proposition \ref{heat Lp} for further details.

We now outline the basic strategy for the global well-posedness theory. In the low-frequency regime, we derive spectrally localized estimates by means of the Duhamel formula together with the sharp semigroup bound \eqref{new kernel}. Consequently, the new low-frequency bounds are consistent with the maximal $L^{p}$-regularity estimates for the heat equation. In the high-frequency regime, we rely on the classical $L^{p}$ energy method, as in standard compressible models. Since both the low- and high-frequency estimates are carried out in $L^{p}$, the nonlinear estimates can be treated by the classical $L^{p}$ product maps, for instance $\dot{B}^{d/p}_{p,1}\times\dot{B}^{d/p-1}_{p,1}$ to  $\dot{B}^{d/p-1}_{p,1}$ ($1\leq p<2d)$, which allows us to avoid resorting to the mixed $L^{2}$–$L^{p}$ framework and its technical index restriction \eqref{pold}.

To establish the large-time behavior of solutions, we develop a new approach to derive optimal decay estimates and remove the smallness assumption on the initial low-frequencies, which essentially consists of the following two steps:
\begin{itemize}
\item[(i)] Uniform propagation of the low-frequency norm;
\item[(ii)] Lyapunov energy estimates and interpolation.
\end{itemize}
Both ingredients are already present in the $L^2$-based setting of Guo and Wang \cite{guo1} and Xu and Xin \cite{xin1}.  The novelty of our result lies in their adaptation to a more general $L^p$ framework $p\in[1,2d)$ and a larger admissible range $\sigma_1\in[\sigma_0-1,\sigma_0)$.

Nevertheless, this strategy brings additional technical difficulties. The first one concerns the restrictions of product laws when $\sigma_1\in[\sigma_0-1,\sigma_0)$. To overcome this issue, we reformulate the system in terms of the momentum variable $m=\rho u$. Then \eqref{1.1} can be rewritten as
\begin{equation}\label{re:momentum}
\left\{
\begin{aligned}
&\partial_t a + \mathrm{div}\,m = 0,\\[1mm]
&\partial_t m
+ \gamma\nabla a+ \kappa\frac{\nabla}{-\Delta} a
- \bar{\mathcal{A}} m
= \dive \mathcal{N},\\
&(a,m)(0,x)=(a_0,m_0)=((\rho_0-1), \rho_0 u_0),
\end{aligned}
\right.
\end{equation}
where the quadratic nonlinear terms are given by
\begin{equation}\label{N}
\begin{aligned}
\mathcal{N}(a,u,m):
&= -u\otimes u-H(a){\rm Id} +\nabla \Lambda^{-2} a\otimes \nabla \Lambda^{-2} a -\frac{1}{2}|\nabla \Lambda^{-2} a|^2 {\rm Id}\\
&\quad+2\tilde\mu_1(a)Dm+\tilde\mu_2(a)\dive m \mathrm{Id}+2(\bar\mu_1+\tilde\mu_1(a))D(Q(a)m)\\
&\quad+(\bar\mu_2+\tilde\mu_2(a))\div(Q(a)m)\mathrm{Id}.
\end{aligned}
\end{equation}
The one-derivative gain exhibited by the nonlinear terms in \eqref{re:momentum} is a key feature when using the localized $L^p$ bound \eqref{new kernel} and Duhamel's principle, as it allows us to circumvent the limitations of classical product laws. A second difficulty arises from the regularity transform between $u$ and $m$. To avoid derivative loss in high frequencies, it is more convenient to work with the velocity  $u$, whereas the low-frequency analysis is naturally performed in terms of the momentum $m$. As a consequence, we first establish estimates for $m$, and then recover the bounds for $u$ through the relation $u = m + Q(a)m$. When $\sigma_0\le \sigma_1<\frac{d}{p}-1$, this conversion can be justified by standard product laws. However, in the range $\sigma_0-1\le \sigma_1<\sigma_0$, such product laws are no longer available, and $u$ and $m$ are no longer equivalent in a straightforward manner. This requires a more delicate analysis to control their interactions.

Finally, regarding methodology, a key novelty of our analysis is {\emph{a time-weighted Lyapunov argument}}. A natural choice of energy would be the same as that used in the global existence theory, namely the $\big(\dot{B}^{d/p-2}_{p,1}\cap \dot{B}^{d/p}_{p,1}\big)\times \dot{B}^{d/p-1}_{p,1}$-norm of $(a,m)$. However, such a choice does not allow us to recover the optimal decay rates for higher-order derivatives. To overcome this difficulty, we refine the energy functional by taking into account the higher-order spatial regularities associated with the $L^{1}_{t}$-type dissipation (see \eqref{1.8}). More precisely, we consider a Lyapunov functional involving the $\dot{B}^{d/p+1}_{p,1}$-norm of $(\Lambda^{-1}a,m)$ in the low-frequency regime. Formally, we get the Lyapunov inequality
\begin{equation}\label{L::}
\begin{aligned}
&\frac{d}{dt}\Big(
\|(\Lambda^{-1}a,m)^{\ell}\|_{\dot{B}^{\frac{d}{p}+1}_{p,1}}
+\|(\Lambda a, u)\|_{\dot{B}^{\frac{d}{p}-1}_{p,1}}^{h}
\Big)
+\|(\Lambda^{-1}a,m)^{\ell}\|_{\dot{B}^{\frac{d}{p}+3}_{p,1}}
+\|(a,\Lambda u)\|_{\dot{B}^{\frac{d}{p}+1}_{p,1}}^{h}  \\
&\qquad\lesssim
G(t)\Big(
\|(\Lambda^{-1}a,m)^{\ell}\|_{\dot{B}^{\frac{d}{p}+1}_{p,1}}
+\|(a,\Lambda u)\|_{\dot{B}^{\frac{d}{p}+1}_{p,1}}^{h}
\Big)\quad\text{with}\quad G(t)\in \mathbb{R}_+.
\end{aligned}
\end{equation}
Consequently, the desired decay estimates are obtained by interpolating between the highest-order dissipation and the lower-order $\dot{B}^{\sigma_1}_{p,\infty}$ norm. 
Moreover, by maximal parabolic regularity, we further improve the decay of $u$ at high frequencies to the $\dot{B}^{d/p+1}_{p,1}$-level. Different from classical approaches, we exploit careful interpolation and Young inequalities in time to bound the nonlinear remainders by an integrable factor $G(t)$ multiplied by the energy itself. This allows us to apply Grönwall’s inequality, rather than relying on smallness to absorb the nonlinear contributions into the dissipation. To make the above formal inequality \eqref{L::} rigorous, we further adapt \eqref{L::} to a time-weighted setting. More precisely, we introduce polynomial time weights of the form $t^{M}$, with $M>0$ arbitrarily large, and derive energy estimates at the level of time integrals rather than differential inequalities, in the spirit of the Fourier splitting method \cite{S-JAMS,schonbekFSM}. It should be emphasized that our approach {\emph{may remain valid even when the critical norms are not small}}. This strategy is of independent interest and is expected to be applicable to other partially dissipative models.

\subsection{Outline and notations}
Finally, the rest of this paper unfolds as follows. In Section 2, we provide a detailed Green's functional analysis to derive new $L^p$ bounds for the linearized system, which is crucial to this paper. In Section 3, we prove the global well-posedness of solutions to the nonlinear problem. Section 4 is devoted to providing the optimal decay estimates. In the last section (``Appendix"), we recall the classical Littlewood-Paley theory and present the proof of local well-posedness in $L^p$-type critical spaces.

Throughout the paper, $C>0$ stands for a generic constant. For brevity, $f\lesssim g$ means that $f\leq Cg$. It will also be understood that $\|(f,g)\|_{X}=\|f\|_{X}+\|g\|_{X}$ for all $f,g\in X$.  Moreover, for $s\in \mathbb{R}$, we define  $\Lambda^{s}f= \mathcal{F}^{-1}(|\xi|^{s}\mathcal{F}(f))$ and simply write $\Lambda=\Lambda^1$.

\section{New \texorpdfstring{$L^p$}{Lp} estimates in linear analysis}\label{section:linear}

From now on, we shall always assume $\kappa=1$ and $\gamma=P'(1)=1$ without loss of generality. In this section, we are going to linearize the perturbed equations and analyze the linear operator in the $L^p$ framework for the following perturbation system \begin{equation}\label{linearized}
\left\{
\begin{aligned}
&\partial_{t}a+\dive u=f,\\
&\partial_{t}u-\bar{\mathcal{A}}u+\nabla a+\frac{\nabla}{-\Delta} a= g,\\
&(a,u)|_{t=0}=(a_{0},u_{0}).
\end{aligned}
\right.
\end{equation}
We establish the following Proposition concerning the uniform estimates for solutions of \eqref{linearized}:
\begin{prop}\label{linear wellposedness}
Let $1 \leq p\leq \infty$, $s \in \mathbb{R}$ and $1 \leq \rho_1 \leq \infty$. 
If $(a,u)$ is a solution to \eqref{linearized} on $[0,T)\times\mathbb{R}^d$, then there exists a constant $C$ depending only on $\bar{\mu}_1$, $\bar{\mu}_2$, $s$ and $d$ such that for $\rho\in[\rho_{1},\infty]$, the following uniform estimates hold:
\begin{itemize}
    \item Low frequency estimates:
   \begin{eqnarray}\label{linear low}
\|(\Lambda^{-1}a,u)\|_{\widetilde{L}_{T}^{\rho}(\dot{B}_{p,1}^{s+\frac{2}{\rho}})}^\ell 
\leq C\Big(\|(\Lambda^{-1}a_0,u_0)\|_{\dot{B}_{p,1}^{s}}+\|(\Lambda^{-1}f,g)\|_{\widetilde{L}_{T}^{\rho_{1}}(\dot{B}_{p,1}^{s-2+\frac{2}{\rho_{1}}})}^\ell\Big),\end{eqnarray}

    \item High frequency estimates:
   \begin{eqnarray}\label{linear high}
\|a\|_{\widetilde{L}_{T}^{\rho}(\dot{B}_{p,1}^{s+1})}^h+\|u\|_{\widetilde{L}_{T}^{\rho}(\dot{B}_{p,1}^{s+\frac{2}{\rho}})}^h 
\leq C\Big(\|(\Lambda a_0,u_0)\|^h_{\dot{B}_{p,1}^{s}}+\|(\Lambda f,g)\|_{\widetilde{L}_{T}^{\rho_{1}}(\dot{B}_{p,1}^{s})}^h\Big).\end{eqnarray} 
\end{itemize}
\end{prop}

We will prove the above Proposition \ref{linear wellposedness}, which relies on a delicate analysis of the Green matrix of the solution in different frequency regimes.

\subsection{Analysis on the Green matrix}\label{Sublinear}

Let $v=(\Lambda^{-1}Ha,u)$ and $v_0=(\Lambda^{-1}Ha_0,u_0)$ be defined with an inhomogeneous operator $H={\rm Id}-\Delta$ measuring the interactions between pressure and electrostatic
potential. It follows from the standard Duhamel formula that the solution of \eqref{linearized} satisfies the following equality:
\begin{equation}v(t)=\mathcal{G}(t)v_0+\int^{t}_0\mathcal{G}(t-s)(\Lambda^{-1}Hf,g)\,ds,\label{Green}
\end{equation}
where $v_0=(\Lambda^{-1}Ha_0,u_0)$, and $\mathcal{G}(t)_{t\geq0}$ are the semigroups associated with the linear system (\ref{linearized}).

To prove Proposition \ref{heat Lp}, we need to analyze the explicit expression of $\mathcal{G}(t)=e^{-tA(D)}$, which is given by the Fourier symbol
\begin{align*}
\widehat{\mathcal{G}}(t,\xi)
&=
\left(
\begin{array}{cc}
\widehat{\mathcal{G}}_{11}&\widehat{\mathcal{G}}_{12}\\
\widehat{\mathcal{G}}_{21}&\widehat{\mathcal{G}}_{22}
\end{array}
\right)
=
\left(
\begin{array}{cc}
p_{1}(|\xi|) & -p_{2}(|\xi|)\,\xi^{\top}\\
-p_{3}(|\xi|)\,\xi & p_0(|\xi|)\mathrm{I}_{\rm{d}}+p_{4}(|\xi|)\,\xi\xi^{\top}
\end{array}
\right),
\end{align*}
where $\mathrm{I}_{\rm{d}}$ is a $d\times d $ unit matrix, and  $p_{i}(|\xi|)$ ($i=0, 1,2,3,4$) are given by
\[
\begin{alignedat}{2}
p_{0}(|\xi|) \; &=\; e^{-\muu_1|\xi|^2 t},\\[1mm]
p_{1}(|\xi|) \; &=\; \frac{\lambda_+ e^{\lambda_- t} - \lambda_- e^{\lambda_+ t}}{\lambda_+ - \lambda_-},\\
p_{2}(|\xi|) \; &=\; -i\,\frac{\widehat H(|\xi|)}{|\xi|}\,
\frac{e^{\lambda_+ t} - e^{\lambda_- t}}{\lambda_+ - \lambda_-},\\[1mm]
p_{3}(|\xi|) \; &=\; -i\,\frac{1}{\widehat H(|\xi|)}
\bigl(|\xi|^3+|\xi|\bigr)\,
\frac{e^{\lambda_+ t} - e^{\lambda_- t}}{\lambda_+ - \lambda_-},\\
p_{4}(|\xi|) \; &=\;
\left(
\frac{\lambda_+ e^{\lambda_+ t} - \lambda_- e^{\lambda_- t}}{\lambda_+ - \lambda_-}
- e^{-\muu_1|\xi|^2 t}
\right)\frac{1}{|\xi|^2}.
\end{alignedat}
\]
Here, the eigenvalues $\lambda_{\pm}(\xi)$ are defined by
$$\lambda_{\pm}(\xi)=\frac{-\muu|\xi|^2\pm\sqrt{\muu^2|\xi|^4-4(1+|\xi|^2)}}{2},$$
with $\bar\mu=2\bar\mu_1+\bar\mu_2$. 

We now establish the following lemma, which provides the key spectrally localized $L^p$ estimates for the semigroup $\mathcal{G}(t)$. 
In particular, the low-frequency estimate is the {\emph{main ingredient}} of this work. 
Compared with the Poisson-free case, the Klein--Gordon-type regularization allows us to work in general $L^p$ spaces.

\begin{prop}\label{heat Lp}
Let $j_0\in\Z$ be an arbitrary (but fixed) integer as the threshold separating the low- and high-frequency regimes.
Let $p\in[1,\infty]$ and let $f$ be any tempered distribution.
Then there exist positive constants $r_0,\widetilde r_0$ (depending only on $d$ and $j_0$) such that,
for all $j\in\Z$, the following $L^p$ estimates hold:
\begin{equation}\label{BB-low-high}
\|\ddj \mathcal{G}(t)f\|_{L^p}
\lesssim
\begin{cases}
e^{-r_0\,2^{2j}t}\,\|\ddj f\|_{L^p}, & j\le j_0\,{\rm(}\text{low frequencies}{\rm)},\\[1mm]
e^{-\widetilde r_0\,t}\,\|\ddj f\|_{L^p}, & j> j_0\,{\rm(}\text{high frequencies}{\rm)}.
\end{cases}
\end{equation}
\end{prop}

The proof of Proposition \ref{heat Lp} mainly relies on the following lemma, which concerns the pointwise behavior of the Fourier symbol for $\mathcal{G}$:

\begin{lem}\label{G}
Let $A_0>0$ and $\alpha\in\mathbb{N}$.
There exist constants $r_0,\widetilde r_0>0$ depending only on $d$ such that, for all $t\ge0$,
all $i,j\in\{1,2\}$ and all $\xi\in\R^d$,
\begin{equation}\label{AA-cases}
\bigl|\nabla_\xi^\alpha \widehat{\mathcal{G}}_{ij}(t,\xi)\bigr|
\;\lesssim\;
|\xi|^{-\alpha}
\begin{cases}
e^{-r_0|\xi|^{2}t}, & |\xi|\le A_0,\\[1mm]
e^{-\widetilde r_0 t}, & |\xi|>A_0.
\end{cases}
\end{equation}
\end{lem}

\begin{proof}
We begin by proving the low-frequency estimate in \eqref{AA-cases}. Without loss of generality, we mostly pay attention to $\widehat{\mathcal{G}}_{11}(t,\xi)$ and assume throughout that $A_0\gg1$. Let  $\xi_0$ be the positive solution of 
$$\frac{\bar{\mu}^2}{4}x^4- x^2-1=0.$$  
In the regime $|\xi|\leq A_0$, we decompose $\widehat{\mathcal{G}}_{11}(t,\xi)$ into the following three regimes with $c_1,c_2>0$:
\begin{itemize}
    \item 
[(1)]. $|\xi|\leq c_1\ll1$: $\chi\left(\frac{|\xi|}{c_1}\right)\widehat{\mathcal{G}}_{11}(t,\xi)$;

    \item 
[(2)]. $\big||\xi|-|\xi_0|\big|\leq c_2\ll1$: $\chi\left(\frac{|\xi|-\xi_0}{c_2}\right)\widehat{\mathcal{G}}_{11}(t,\xi)$;

    \item 
[(3)]. Otherwise: $\left(1-\chi\left(\frac{|\xi|}{c_1}\right)-\chi\left(\frac{|\xi|-\xi_0}{c_2}\right)\right)\widehat{\mathcal{G}}_{11}(t,\xi)$. 
\end{itemize}

Each regime can be addressed as follows.

\begin{itemize}

\item \underline{Case 1: $ |\xi|\leq c_1\ll1$}.
\end{itemize}
We start with the pointwise estimates for $\widehat{\mathcal{G}}_{11}$ under $ |\xi|\leq c_1$ with some $c_1\ll1$. Since $\frac{\bar{\mu}^2}{4}|\xi|^4-1-|\xi|^2<0$, we have
$$\lambda_{\pm}=-\frac{\bar{\mu}}{2}|\xi|^2\pm iB(\xi)\quad \text{with}\quad B(\xi)=\sqrt{1+|\xi|^2-\frac{\bar{\mu}^2}{4}|\xi|^4}$$
and further
$$|\widehat{\mathcal{G}}_{11}|\lesssim\left|\frac{\lambda_+e^{\lb t}}{B(\xi)}\right|+\left|\frac{\lb e^{\la t}}{B(\xi)}\right|
$$

Then, (\ref{AA-cases}) with $\alpha=0$ is directly obtained once we notice $|\la(\xi)|\sim|B(\xi)|\sim1$.
To prove the  derivative estimate of the semigroup in Fourier space, the {\emph{key point}} is the following bounds:
\begin{align}
    |\partial^a_{r}B(r)|\leq r^{2-a}\quad\text{for any integer}\,   a\geq1\quad\text{and}\quad  r\leq c_1\ll1.\label{claim}
\end{align}
To prove \eqref{claim}, we set $F(r)=B(r)^2=1+ r^2-\frac{\bar\mu^2}{4}r^4$. Choosing $c_1>0$ small, we have $F(r)\ge\frac12$ on $[0,c_1]$.
Moreover, 
\begin{align*}
\partial_r^1 F(r)&=2 r-\bar\mu^2 r^3=O(r),\qquad\,\,\,
\partial_r^2 F(r)=2-3\bar\mu^2 r^2=O(1),\\
\partial_r^3 F(r)&=-6\bar\mu^2 r=O(r),\qquad\quad \quad 
\partial_r^4 F(r)=-6\bar\mu^2=O(1),
\end{align*}
and $\partial_r^5 F\equiv0$ for $k\ge5$.
By the Faà di Bruno formula, each $\partial_r^a B$
is a finite sum of terms of the form
\[
F(r)^{\frac12-\ell}\prod_{k=1}^{4}\bigl(\partial_r^k F(r)\bigr)^{\alpha_k},
\qquad
\alpha_1+2\alpha_2+3\alpha_3+4\alpha_4=a,
\qquad
\ell=\alpha_1+\alpha_2+\alpha_3+\alpha_4\ge1.
\]
Using $F^{\frac12-\ell}\lesssim1$ on $[0,c_1]$ and the order of $\partial^k F(r)$ ($k=1,2,3,4$),
we infer that each term satisfies
\[
\left|F(r)^{\frac12-\ell}\prod_{k=1}^{4}(\partial_r^kF(r))^{\alpha_k}\right|
\lesssim r^{\alpha_1+\alpha_3}.
\]
Note that $\alpha_1+\alpha_3\geq 2-a$ holds. Indeed, for $a\ge2$, one has $\alpha_1+\alpha_3\geq 0\ge 2-a$.
The case $a=1$ is trivial, since necessarily $\alpha_1=1$ and $\alpha_3=0$.
Hence,  summing up finitely many such terms yields  \eqref{claim}.

Now let $\alpha=1$, then apparently when the derivative lands on  heat kernels $e^{-\frac{\bar{\mu}}{2}|\xi|^2t}$, we easily get the first inequality in (\ref{AA-cases}). Now let us compute when the derivative lands on polynomials
\begin{eqnarray}
\left|\partial_{\xi_{i}}\frac{\lambda_{+}(\xi)}{B(\xi)}\right|=\left|\frac{\partial_r\lambda_{+}(\xi)B(\xi)-\lambda_{+}(\xi)\partial_r B(\xi)}{B^2(\xi)}\frac{\xi_{i}}{|\xi|}\right|\lesssim|\xi|^{-1}.
\end{eqnarray}
Finally, we consider derivatives landing on the conjugate part. Indeed, we have
\begin{eqnarray}
|\partial_{\xi_{i}}e^{iB(\xi)t}|=\left|e^{iB(\xi)t}t\partial_r B(|\xi|)\frac{\xi_{i}}{|\xi|}\right|\lesssim |\xi|t,
\end{eqnarray}
then we could utilize the heat kernel $|\xi|te^{-\frac{\bar{\mu}}{2}|\xi|^2t}\leq|\xi|^{-1}e^{- r_0|\xi|^2t}$  with $r_0\ll1$ to arrive
\begin{eqnarray*}
\left|\nabla_{\xi}\Big(\frac{\lambda_{+}(\xi)}{2iB(\xi)}e^{\lambda_{-}(\xi)t}\Big)\right|\lesssim |\xi|^{-1}e^{- r_0|\xi|^2 t}.
\end{eqnarray*}
We can symmetrically handle the other term. Similarly, one could generalize the above estimate to higher-order derivatives of $\xi$ and we get (\ref{AA-cases}) for $|\xi|\leq A_0$.

\begin{itemize}

\item \underline{Case 2: $\big||\xi|-|\xi_0|\big|\leq c_2\ll1$}.
\end{itemize}

In this case, without loss of generality, we assume $\frac{\muu^2}{4}|\xi|^4-|\xi|^2-1>0$. Notice that $\big||\xi|-|\xi_0|\big|\leq c_2\ll1$ naturally implies $\tilde B(\xi)\ll1$. Hence, with the aid of Taylor's extension
$$\tilde B(\xi)=\sum_{k\in\N^+} d_k(|\xi|-\xi_0)^k,$$
there holds
\begin{eqnarray*}
\chi\Big(\frac{|\xi|-\xi_0}{c_2}\Big)\widehat{\mathcal{G}}_{11}(\xi)&=&\chi\Big(\frac{|\xi|-\xi_0}{c_2}\Big)\left(e^{\lb t}-e^{-\muu|\xi|^2t}\frac{\lb(e^{-\tilde B(\xi) t}-e^{\tilde B(\xi) t})}{2\tilde B(\xi)}\right)\\
\nonumber
&=&\chi\Big(\frac{|\xi|-\xi_0}{c_2}\Big)\left(e^{\lb t}-\lb\Big(2\tilde B(\xi)+\frac{2}{3!}\tilde B^2(\xi)+...\Big)e^{-\muu|\xi|^2t}\right)\\
\nonumber
&=&\chi\Big(\frac{|\xi|-\xi_0}{c_2}\Big)\left(e^{\lb t}-\lb e^{-\muu|\xi|^2t}\sum_{i\in\N^+} \tilde d_i(|\xi|-\xi_0)^i\right),  
\end{eqnarray*}
which implies
$$\Big|\chi\Big(\frac{|\xi|-\xi_0}{c_2}\Big)\widehat{\mathcal{G}}_{11}(\xi)\Big|\lesssim e^{-c|\xi_0|^2t}\lesssim e^{-c' t}\lesssim e^{-c'' |\xi|^2 t}.$$
for some small constants $c, c, c''>0$. Furthermore, in terms of the derivative's estimate, we only present when the derivative lands on a localized function, where
\begin{eqnarray}
\left|\partial_{\xi_{i}}\chi\Big(\frac{|\xi|-\xi_0}{c_2}\Big)\right|=\left|c_2^{-1}\chi'\Big(\frac{|\xi|-\xi_0}{c_2}\Big)\frac{\xi_i}{|\xi|}\right|\leq c_{\xi_0}|\xi|^{-1},
\end{eqnarray}
since $\xi_0$ is the fixed constant. Imposing similar estimates on higher order derivatives, we obtain \eqref{AA-cases} for Case 2.


\begin{itemize}
\item \underline{Case 3: Otherwise: $c_1\le |\xi|\le \xi_0-c_2$ or $\xi_0+c_2\le |\xi|\le A_0$. }
\end{itemize}
In this case, there exists $\widetilde r_1>0$
such that
\begin{equation}\label{case3-gap}
\Re\lambda_\pm(\xi)\le -\widetilde r_1,
\qquad
|\lambda_\pm(\xi)|\ge \widetilde r_1,
\qquad
B(\xi)\ge \widetilde r_1,
\qquad
\tilde B(\xi)\ge \widetilde r_1.
\end{equation}
Consequently, all coefficients appearing in the explicit formula of
$\widehat{\mathcal{G}}_{11}(t,\xi)$ (and similarly for the other entries)
are smooth and uniformly bounded with respect to $\xi$, and repeated differentiation
in $\xi$ only produces finite sums of terms of the form
\[
t^m\,Q_{\alpha,m}(\xi)\,e^{\lambda_\pm(\xi)t},
\qquad 0\le m\le|\alpha|,
\]
where $Q_{\alpha,m}$ is smooth and bounded (depending on $\alpha$ and $m$) when $\xi$ lies in this regime.
Using \eqref{case3-gap} and $t^m e^{-\widetilde r_1 t}
\lesssim e^{-\frac12\widetilde r_1 t}$, we infer that for all multi-indices
$\alpha$,
\[
\bigl|\nabla_\xi^\alpha \widehat{\mathcal{G}}_{11}(t,\xi)\bigr|
\lesssim e^{-\widetilde r_1 t}\lesssim |\xi|^{-|\alpha|}e^{-\widetilde c_1 |\xi|^2 t}.
\]
with some small constant $\widetilde c_1>0$. The other components $\widehat{\mathcal{G}}_{ij}$ are treated in the same way. We thus complete the proof of \eqref{AA-cases} when $|\xi|\leq A_0$.


Finally, for the high frequencies $|\xi|\geq A_0$, we have  $\frac{\muu^2}{4}|\xi|^4-|\xi|^2-1\gg1$.  We still focus on $\widehat{\mathcal{G}}_{11}$. At this stage, we write $\tilde B(\xi)=\sqrt{\frac{\muu^2}{4}|\xi|^4-|\xi|^2-1}$ and have the following facts:
$$|\lb(\xi)|\sim|\xi|^2,\quad
|\la(\xi)|=\left|\frac{2(1+|\xi|^2)}{\lb(\xi)}\right|\sim1,\quad|\la-\lb|\sim|B(\xi)|\sim|\xi|^2,$$
which implies
$$|\widehat{\mathcal{G}}_{11}|\lesssim|\xi|^{-2}\big(e^{-|\xi|^{2} t}+|\xi|^{2}e^{- t}\big)\lesssim e^{-\tilde r_0 t},$$
where $\tilde r_0$ is a sufficiently small constant. The derivative estimates could be bounded in the same fashion as low frequencies, and we are led to  the second inequality in (\ref{AA-cases}). This concludes Lemma \ref{G}.

\end{proof}

\noindent 
\textbf{Proof of Proposition \ref{heat Lp}.}
We now prove Proposition \ref{heat Lp}. It suffices to establish, for $j\le j_0$,
the kernel estimate
\begin{eqnarray}\label{kernel}
\Big\|\mathcal{F}^{-1}\Big(\widehat{\mathcal{G}_{ij}}(\xi)\varphi(\xi/2^j)\Big)\Big\|_{L^1}\lesssim  e^{- r_02^{2j}t}
\end{eqnarray}
since the $L^p$ bound in \eqref{BB-low-high} then follows from Young's inequality. Here $r_0>0$ can be chosen to be suitably small.  

For simplicity, we only consider the component $\widehat{\mathcal{G}}_{11}$.  We begin with $|x|\leq2^{-j}$. Direct computations indicate
\begin{equation}
\begin{aligned}\label{pointwise1}
\Big\|\mathcal{F}^{-1}\Big(\widehat{\mathcal{G}}_{11}(\xi)\varphi(\xi/2^j)\Big)\Big\|_{L^1_{|x|\le 2^{-j}}}
&=\Big\|\int_{\R^d} e^{ix\cdot\xi}\widehat{\mathcal{G}}_{11}(\xi)\varphi(\xi/2^j)\,d\xi\Big\|_{L^1_{|x|\le 2^{-j}}}\\
&\lesssim |B(0,2^{-j})| \sup_{|\xi|\sim 2^j}|\widehat{\mathcal{G}}_{11}(\xi)|
\|\varphi(\xi/2^j)\|_{L^1_\xi}\\
&\lesssim  e^{-r_02^{2j}t},
\end{aligned}
\end{equation}
where we used $\|\varphi(\xi/2^j)\|_{L^1_\xi}\sim 2^{dj} \|\varphi\|_{L^1_{\xi}}\lesssim 2^{dj}$ and \eqref{BB-low-high} when $j\leq j_{0}$.

Next, we turn to address $|x|\geq2^{-j}$. Using the identity $e^{ix\cdot\xi}=-\frac{1}{|x|^2}\Delta_{\xi}e^{ix\cdot\xi}$ and integration by parts, we arrive at
\begin{eqnarray}\label{pointwisefff}
\int_{\R^d}e^{ix\cdot\xi}\widehat{\mathcal{G}_{11}}(\xi)\varphi(\xi/2^j)\,d\xi
= |x|^{-2k}\int_{\R^d}e^{ix\cdot\xi}\Delta^{k}_{\xi}\Big(\widehat{\mathcal{G}_{11}}(\xi)\varphi(\xi/2^j)\Big)\,d\xi
\end{eqnarray}
 Now, in light of Lemma \ref{G}, we could continue with (\ref{pointwisefff}) and obtain
\begin{align}\label{pointwise}
 \Big\|\mathcal{F}^{-1}\Big(\widehat{\mathcal{G}_{11}}(\xi)\varphi(\xi/2^j)\Big)\Big\|_{L^1_{|x|\geq2^{-j}}}\nonumber&=
\Big\||x|^{-2k}\int_{\R^d}e^{ix\cdot\xi}\Delta^{k}_{\xi}\Big(\widehat{\mathcal{G}_{11}}(\xi)\varphi(\xi/2^j)\Big)\,d\xi\Big\|_{L^1_{|x|\geq2^{-j}}}\nonumber \\
&\lesssim 2^{-dj}e^{- r_0 2^{2j}t}\|\varphi(\xi/2^j)\|_{L^1_{\xi}}\lesssim e^{- r_02^{2j}t}
\end{align}
where we used $\| |x|^{-2k}\|_{L^1_{|x|\ge 2^{-j}}}\lesssim 2^{(2k-d)j}$ and that the growth in $\xi$
is compensated by derivatives falling on $\varphi(\xi/2^j)$. Consequently, combining \eqref{pointwise1}, \eqref{pointwise}, we have \eqref{kernel}.

Therefore, for $j\leq j_0$, it holds that
\begin{equation}\|\ddj \mathcal{G} (t)f\|_{L^p}\lesssim \sum_{i,j=1,2}\Big\|\mathcal{F}^{-1}\Big(\widehat{\mathcal{G}_{ij}}(\xi)\varphi(\xi/2^j)\Big)\Big\|_{L^1}\|\ddj f\|_{L^p}\lesssim e^{- r_02^{2j}t}\|\ddj f\|_{L^p}\end{equation}
which yields the first inequality in (\ref{BB-low-high}). The high-frequency exponential decay estimate in (\ref{BB-low-high}) can be proved in the same manner. This completes the proof of Proposition~\ref{heat Lp}.

\subsection{Proof of Proposition \ref{linear wellposedness}}

We now prove the linear estimates for \eqref{linearized}.
By Duhamel's formula, we have
\begin{equation*}
(\Lambda^{-1}Ha,u)
=\mathcal{G}(t)(\Lambda^{-1}Ha_0,u_0)
+\int_0^t \mathcal{G}(t-s)(\Lambda^{-1}Hf,g)\,ds .
\end{equation*}
For the homogeneous part, applying Fourier localization and Lemma~\ref{heat Lp},
we obtain for $j\le j_0$,
\begin{equation*}
\|\dot\Delta_j \mathcal{G}(t)(\Lambda^{-1}a_0,u_0)\|_{L^p}
\lesssim e^{-r_0 2^{2j}t}\,
\|\dot\Delta_j(\Lambda^{-1}a_0,u_0)\|_{L^p},
\end{equation*}
where we used the fact that $H(\xi)\sim1$ for $|\xi|\lesssim 2^{j_0}$.
Taking the Besov norm yields
\begin{equation}
\|\mathcal{G}(t)(\Lambda^{-1}a_0,u_0)\|_{\widetilde L_T^{\rho}
(\dot B_{p,r}^{s+\frac{2}{\rho}})}^\ell
\lesssim
\|(\Lambda^{-1}a_0,u_0)\|_{\dot B_{p,r}^{s}}^\ell.
\end{equation}
The source term is handled similarly by Young's inequality, which gives
\begin{equation}
\|(\Lambda^{-1}a,u)\|_{\widetilde L_T^{\rho}
(\dot B_{p,r}^{s+\frac{2}{\rho}})}^\ell
\lesssim
\|(\Lambda^{-1}a_0,u_0)\|_{\dot B_{p,r}^{s}}^\ell
+\|(\Lambda^{-1}f,g)\|_{\widetilde L_T^{\rho_1}
(\dot B_{p,1}^{s-2+\frac{2}{\rho_1}})}^\ell .
\end{equation}
For high frequencies, Lemma~\ref{heat Lp} implies
\begin{equation}\label{M3}
\|(\Lambda^{-1}Ha,u)\|_{\widetilde L_T^{\rho}
(\dot B_{p,r}^{s})}^h
\lesssim
\|(\Lambda^{-1}Ha_0,u_0)\|_{\dot B_{p,r}^{s}}^h
+\|(\Lambda^{-1}Hf,g)\|_{\widetilde L_T^{\rho_1}
(\dot B_{p,1}^{s})}^h .
\end{equation}

Finally, we derive the smoothing effect for the velocity.
Indeed, $u$ satisfies
\begin{equation*}
\partial_t u-\bar{\mathcal A}u
=-\nabla a-\frac{\nabla}{-\Delta}a+g,
\end{equation*}
hence
\[
u=e^{\bar{\mathcal A}t}u_0
-\int_0^t e^{\bar{\mathcal A}(t-s)}
\Big(\nabla a+\frac{\nabla}{-\Delta}a+g\Big)\,ds .
\]
The standard parabolic smoothing estimate for the Lam\'e operator gives
\begin{equation}\label{M4}
\|u\|_{\widetilde L_T^{\rho}
(\dot B_{p,r}^{s+\frac{2}{\rho}})}^h
\lesssim
\|u_0\|_{\dot B_{p,r}^{s}}^h
+\Big\|\nabla a+\frac{\nabla}{-\Delta}a\Big\|_{\widetilde L_T^{\rho_1}
(\dot B_{p,1}^{s-2+\frac{2}{\rho_1}})}^h
+\|g\|_{\widetilde L_T^{\rho_1}
(\dot B_{p,1}^{s-2+\frac{2}{\rho_1}})}^h .
\end{equation}
Since
\[
\nabla a+\frac{\nabla}{-\Delta}a
\sim \nabla a \sim \Lambda^{-1}Ha
\quad \text{in high frequencies},
\]
we have
\[
\Big\|\nabla a+\frac{\nabla}{-\Delta}a\Big\|_{\widetilde L_T^{\rho_1}
(\dot B_{p,1}^{s-2+\frac{2}{\rho_1}})}^h
\lesssim
\|\Lambda^{-1}Ha\|_{\widetilde L_T^{\rho_1}
(\dot B_{p,1}^{s})}^h .
\]
Combining this bound with \eqref{M3} and \eqref{M4} yields
\begin{multline}\label{high estimates}
\|(\Lambda^{-1}Ha,u)\|_{\widetilde L_T^{\rho}
(\dot B_{p,r}^{s})}^h
+\|u\|_{\widetilde L_T^{\rho}
(\dot B_{p,r}^{s+\frac{2}{\rho}})}^h
\lesssim
\|(\Lambda^{-1}Ha_0,u_0)\|_{\dot B_{p,r}^{s}}^h\\
+\|(\Lambda^{-1}Hf,g)\|_{\widetilde L_T^{\rho_1}
(\dot B_{p,1}^{s})}^h
+\|g\|_{\widetilde L_T^{\rho_1}
(\dot B_{p,1}^{s})}^h .
\end{multline}
This completes the proof of Proposition~\ref{linear wellposedness}.



\section{Global well-posedness: Proof of Theorem \ref{thm2.1}}\setcounter{equation}{0}


This section is devoted to the proof of the global well-posedness of $L^p$ solutions to the reformulated Cauchy problem \eqref{linearized:1}.

\subsection{A priori estimate}

We define the energy-dissipation functional
\begin{equation}\label{Xt}
\begin{aligned}
\mathcal{X}_p(t):&=\|a\|_{\widetilde{L}^\infty_t(\dot{B}^{\frac{d}{p}-2}_{p,1})}^{\ell}+\|a\|_{\widetilde{L}^\infty_t(\dot{B}^{\frac{d}{p}}_{p,1})}^h+\|a\|_{ L^1_t(\dot{B}^{\frac{d}{p}}_{p,1})}+\|u\|_{\widetilde{L}^\infty_t(\dot{B}^{\frac{d}{p}-1}_{p,1})}+\|u\|_{ L^1_t(\dot{B}^{\frac{d}{p}+1}_{p,1})}.
\end{aligned}
\end{equation}
The proof of global existence relies on uniform \emph{a priori} estimates stated below. Our method is built upon a genuinely new $L^{p}$ semigroup estimate developed in Section~\ref{section:linear}, 
which goes beyond the classical symmetric hyperbolic structure. 

\begin{prop}\label{prop31}
Let $p\in [1,2d)$, and let $(a,u)$ be the solutions of \eqref{linearized:1} on $[0,T)\times\mathbb{R}^d$. If for any $t\in[0,T)$,  
\begin{align}
\|a\|_{L^{\infty}_t(L^{\infty})}\leq 1\label{3.2}
\end{align}
holds, then we have 
\begin{align}\label{3.3}
\mathcal{X}_p(t)\leq C^* \mathcal{X}_p(0)+C^*\big(1+\mathcal{X}_p(t)\big)\mathcal{X}_p(t)^2.
\end{align}
Here $C^*>0$ is a constant independent of $T$.
\end{prop}


\begin{proof}
We divide the proof of Proposition \ref{prop31} into the following three steps.

\begin{itemize}
\item {\emph{Step 1: Linear $L^p$ estimates in high frequencies.}}
\end{itemize}

Recall the incompressible--compressible decomposition $u=\cP u+\cQ u$.
The incompressible part $\cP u$ satisfies
\[
\left\{
\begin{aligned}
&\partial_t\cP u- \bar\mu_1\Delta\cP u=\cP g,\\
&\cP u|_{t=0}=\cP u_{0}.
\end{aligned}
\right.
\]
Hence, the optimal smoothing effect in Lemma~\ref{lemma6.1} yields
\begin{equation}\label{est:Pu}
\|\cP u\|_{\widetilde{L}^{\infty}_{t}(\dot{B}^{\frac{d}{p}-1}_{p,1})\cap L^{1}_{t}(\dot{B}^{\frac{d}{p}+1}_{p,1})}^h
\lesssim
\|\cP u_{0}\|_{\dot{B}^{\frac{d}{p}-1}_{p,1}}^h
+\|\cP g\|_{L^{1}_{t}(\dot{B}^{\frac{d}{p}-1}_{p,1})}^h.
\end{equation}

The compressible subsystem, consisting of the compressible part $v=\cQ u$ and the density fluctuation $a=\rho-1$, satisfies the coupled system
\begin{equation}\label{linearized:111}
\left\{
\begin{aligned}
&\partial_{t}a+u\cdot\nabla a+\dive v=-a\,\dive u ,\\
&\partial_{t}v-\bar\mu\Delta v+\nabla a+\nabla \Lambda^{-2} a= \cQ g,\\
&(a,v)|_{t=0}=(a_{0},v_{0}).
\end{aligned}
\right.
\end{equation}
with $\bar\mu=2\bar\mu_1+\bar\mu_2$.

Note that the Fourier symbol of the linear operator associated with \eqref{linearized:111} has purely real eigenvalues. This hyperbolic-parabolic structure enables an energy decoupling argument after choosing a suitable frequency threshold $j_0$.
Following Hoff~\cite{hoff97} and Haspot~\cite{haspot1}, one look a  effective velocity  such that $-\bar\mu\Delta v+\nabla a+\nabla \Lambda^{-2} a=\bar\mu\Lambda^{2}w$, i.e.,
\begin{equation}\label{def:w}
w
=v+\frac{1}{\bar\mu}\nabla \Lambda^{-2}\bigl(a+\Lambda^{-2}a\bigr).
\end{equation}
Substituting $v=w-\frac{1}{\bar\mu}\nabla \Lambda^{-2}\bigl(a+\Lambda^{-2}a\bigr)$ into the first equation of \eqref{linearized:111} yields the damped-type equation
\begin{equation}\label{eq:a-damped}
\partial_t a+u\cdot\nabla a+\frac{1}{\bar\mu}\,a
=
-\frac{1}{\bar\mu}\,\Lambda^{-2}a-\dive w-a\,\dive u .
\end{equation}
Performing the localized $L^p$ energy estimate for \eqref{eq:a-damped}, we have
\begin{align*}
&\quad\|\dot{\Delta}_ja\|_{L^{\infty}_t(L^p)}+\frac{1}{\bar\mu}\|\dot{\Delta}_ja\|_{L^{1}_t(L^p)}\\
&\leq \|\dot{\Delta}_ja_0\|_{L^p}+\int_0^t \Big(\frac{1}{p}\|\dive u\|_{L^\infty}\|\dot{\Delta}_ja\|_{L^p}+\frac{1}{\bar\mu}\|\Lambda^{-2}\dot{\Delta}_j a\|_{L^p}+\|\dive \dot{\Delta}_j w\|_{L^p}+\|\dot{\Delta}_j(a\dive u)\|_{L^p}\Big)\,d\tau.
\end{align*}
Thus, the classical commutator estimate, together with the embedding $\dot{B}^{\frac{d}{p}}_{p,1}\hookrightarrow L^\infty$ and Bernstein's inequality, indicates that
\begin{equation}\label{48}
\begin{aligned}
\|a\|_{\widetilde{L}^{\infty}_{t}(\dot{B}^{\frac{d}{p}}_{p,1})}^{h}
+\frac{1}{\bar\mu}\|a\|_{L^{1}_{t}(\dot{B}^{\frac{d}{p}}_{p,1})}^{h}
&\leq
\|a_{0}\|_{\dot{B}^{\frac{d}{p}}_{p,1}}^{h}
+\|w\|_{L^{1}_{t}(\dot{B}^{\frac{d}{p}+1}_{p,1})}^{h}
+C_1 2^{-2j_0}\|a\|_{L^{1}_{t}(\dot{B}^{\frac{d}{p}}_{p,1})}^{h}\\
&\quad
+C_1\|\nabla u\|_{L^1_t(\dot{B}^{\frac{d}{p}}_{p,1})}\,
\|a\|_{\widetilde{L}^{\infty}_{t}(\dot{B}^{\frac{d}{p}}_{p,1})}^{h}
+\|a\,\dive u\|_{L^{1}_{t}(\dot{B}^{\frac{d}{p}}_{p,1})}^{h}.
\end{aligned}
\end{equation}

On the other hand, differentiating \eqref{def:w} in time and using the second equation of \eqref{linearized:1}, we infer that $w$ satisfies
\begin{equation}\label{eq:w-parabolic}
\partial_t w-\bar\mu\Delta w
= \cQ g
+\frac{1}{\bar\mu}\bigl(\nabla\Lambda^{-2}+\nabla\Lambda^{-4}\bigr)
\Bigl(-\div w-\div(a u)+\frac{1}{\bar\mu}\,a+\frac{1}{\bar\mu}\,\Lambda^{-2}a
\Bigr).
\end{equation}

Then, applying Lemma~\ref{lemma6.1} yields
\begin{equation}\label{410}
\begin{aligned}
\|w\|_{\widetilde{L}^{\infty}_{t}(\dot{B}^{\frac{d}{p}-1}_{p,1})}^{h}
+c\|w\|_{L^{1}_{t}(\dot{B}^{\frac{d}{p}+1}_{p,1})}^{h}
&\leq
\|w_{0}\|_{\dot{B}^{\frac{d}{p}-1}_{p,1}}^{h}
+C_2 2^{-2j_0}(1+2^{-2j_0})
\|w\|_{L^{1}_{t}(\dot{B}^{\frac{d}{p}+1}_{p,1})}^{h}\\
&\quad
+C_2 2^{-j_{0}}(1+2^{-4j_{0}})
\|a\|_{L^{1}_{t}(\dot{B}^{\frac{d}{p}}_{p,1})}^{h}\\
&\quad
+C_2 (2^{-j_0}+2^{-3j_0})\|a u\|_{L^{1}_{t}(\dot{B}^{\frac{d}{p}}_{p,1})}
+\|g\|_{L^{1}_{t}(\dot{B}^{\frac{d}{p}-1}_{p,1})}^{h}.
\end{aligned}
\end{equation}

Let us take $\frac{c}{2}\times \eqref{48}+\eqref{410}$ and choose $j_0$ large enough so that
\begin{equation}\label{choose:j0}
C_1 2^{-2j_0}\leq \frac{1}{16},
\qquad
C_2 2^{-2j_0}(1+2^{-2j_0})\leq \frac{c}{4},
\qquad
C_2 2^{-j_{0}}(1+2^{-4j_{0}})\leq \frac{c}{16\bar\mu}.
\end{equation}
Consequently, combining the above estimates with \eqref{est:Pu}, using the relation $
v=w-\frac{1}{2}\nabla \Lambda^{-2}\bigl(a+\Lambda^{-2}a\bigr)$  to recover $v$ from the bounds of $(a,w)$, we end up with
\begin{equation}\label{high:es}
\begin{aligned}
&\quad \|a\|_{\widetilde{L}^{\infty}_{t}(\dot{B}^{\frac{d}{p}}_{p,1})\cap L^{1}_{t}(\dot{B}^{\frac{d}{p}}_{p,1})}^{h}
+\|v\|_{\widetilde{L}^{\infty}_{t}(\dot{B}^{\frac{d}{p}-1}_{p,1})\cap L^{1}_{t}(\dot{B}^{\frac{d}{p}+1}_{p,1})}^{h}\\
&\lesssim
\|a_{0}\|_{\dot{B}^{\frac{d}{p}}_{p,1}}^{h}
+\|v_{0}\|_{\dot{B}^{\frac{d}{p}-1}_{p,1}}^{h}
+\| u\|_{L^1_t(\dot{B}^{\frac{d}{p}+1}_{p,1})}
\|a\|_{\widetilde{L}^{\infty}_{t}(\dot{B}^{\frac{d}{p}}_{p,1})}^{h}
+\|a\,\dive u\|_{L^{1}_{t}(\dot{B}^{\frac{d}{p}}_{p,1})}^{h}\\
&\quad
+\|a u\|_{L^{1}_{t}(\dot{B}^{\frac{d}{p}}_{p,1})}^h
+\|g\|_{L^{1}_{t}(\dot{B}^{\frac{d}{p}-1}_{p,1})}^{h}.
\end{aligned}
\end{equation}

\begin{itemize}
\item {\emph{Step 2: Linear $L^p$ estimates in low frequencies.}}
\end{itemize}

We are now in a position to establish $L^p$ estimates in low frequencies, which
constitutes a substantial improvement compared with the classical $L^2$
framework in the literature. The frequency threshold $j_0$
has already been fixed in Step~1. Applying \eqref{linear low} in Proposition \ref{linear wellposedness} with $s=\frac{d}{p}-1$ directly yields
\begin{equation}\label{Lp-lowfreq}
\|(\Lambda^{-1}a,u)\|_{\widetilde{L}_{t}^{\infty}(\dot{B}_{p,1}^{\frac{d}{p}-1})\cap L_{t}^{1}(\dot{B}_{p,1}^{\frac{d}{p}+1})}^{\ell}
\lesssim
\|(\Lambda^{-1}a_0,u_0)\|^\ell_{\dot{B}_{p,1}^{\frac{d}{p}-1}}
+\|au\|_{\widetilde{L}_{t}^{1}(\dot{B}_{p,1}^{\frac{d}{p}-1})}+ \|g\|^\ell_{\widetilde{L}_{t}^{1}(\dot{B}_{p,1}^{\frac{d}{p}-1})}.
\end{equation}

\begin{itemize}
\item \emph{Step 3: Nonlinear estimates.}
\end{itemize}

We consider nonlinear estimates in \eqref{high:es} and \eqref{Lp-lowfreq}.
In light of the second product law in Proposition \ref{prop3.2}, we have for $p\in[1,2d)$ that
\begin{eqnarray*}
 \|au\|_{L^{1}_{t}(\dot{B}^{\frac{d}{p}-1}_{p,1})}
\lesssim\|a\|_{L^{1}_{t}(\dot{B}^{\frac{d}{p}}_{p,1})}
\|u\|_{\widetilde{L}^{\infty}_{t}(\dot{B}^{\frac{d}{p}-1}_{p,1})}\lesssim \mathcal{X}^2_p(t),
\end{eqnarray*}
and
\begin{align*}
\|au\|_{L^{1}_{t}(\dot{B}^{\frac{d}{p}}_{p,1})}\lesssim\|a\|_{\widetilde{L}^{2}_{t}(\dot{B}^{\frac{d}{p}}_{p,1})}
\| u\|_{\widetilde{L}^{2}_{t}(\dot{B}^{\frac{d}{p}}_{p,1})}\lesssim \mathcal{X}^2_p(t),
\end{align*}
where we used the interpolation inequalities
\begin{align*}
\|a\|_{\widetilde{L}^{2}_{t}(\dot{B}^{\frac{d}{p}}_{p,1})}\lesssim \|a\|_{\widetilde{L}^{\infty}_{t}(\dot{B}^{\frac{d}{p}}_{p,1})}^{\frac{1}{2}} \|a\|_{L^{1}_{t}(\dot{B}^{\frac{d}{p}}_{p,1})}^{\frac{1}{2}}\quad \text{and}\quad \| u\|_{\widetilde{L}^{2}_{t}(\dot{B}^{\frac{d}{p}}_{p,1})}\lesssim \| u\|_{\widetilde{L}^{\infty}_{t}(\dot{B}^{\frac{d}{p}-1}_{p,1})}^{\frac{1}{2}} \| u\|_{L^{1}_{t}(\dot{B}^{\frac{d}{p}+1}_{p,1})}^{\frac{1}{2}}.
\end{align*}
Similarly, one has
\begin{eqnarray*}
\|a\mathrm{div}u\|^h_{L^{1}_{t}(\dot{B}^{\frac{d}{p}}_{p,1})}
\lesssim\|a\|_{\widetilde{L}^{\infty}_{t}(\dot{B}^{\frac{d}{p}}_{p,1})}
\|\mathrm{div}u\|_{L^{1}_{t}(\dot{B}^{\frac{d}{p}}_{p,1})}\lesssim \mathcal{X}^2_p(t).
\end{eqnarray*}

We now deal with the nonlinear term involving $g=g_1+g_2+g_3$. One analyzes $g_1$ as follows: 
\begin{eqnarray*}
\|g_1\|_{L^{1}_{t}(\dot{B}^{\frac{d}{p}-1}_{p,1})}
\lesssim\|u\|_{\widetilde{L}^{\infty}_{t}(\dot{B}^{\frac{d}{p}-1}_{p,1})}
\|\nabla u\|_{L^{1}_{t}(\dot{B}^{\frac{d}{p}}_{p,1})}\lesssim \mathcal{X}^2_p(t).
\end{eqnarray*}
Concerning $g_2$ and $g_3$, by employing Proposition \ref{prop3.2}, \eqref{3.2} and the continuity of composite functions in Proposition \ref{prop2.25} to verify
\begin{align*}
\|g_2\|_{L^{1}_{t}(\dot{B}^{\frac{d}{p}-1}_{p,1})}
&\lesssim(1+\|\tilde{Q}(a)\|_{\tilde{L}^{\infty}_{t}(\dot{B}^{\frac{d}{p}}_{p,1})})\|(\tilde{\mu}_1(a),\tilde{\mu}_2(a))\|_{\tilde{L}^{\infty}_{t}(\dot{B}^{\frac{d}{p}}_{p,1})}
\|\nabla u\|_{L^{1}_{t}(\dot{B}^{\frac{d}{p}}_{p,1})}\\
&\lesssim (1+\|a\|_{\widetilde{L}^{\infty}_{t}(\dot{B}^{\frac{d}{p}}_{p,1})}) \|a\|_{\widetilde{L}^{\infty}_{t}(\dot{B}^{\frac{d}{p}}_{p,1})} \|u\|_{L^{1}_{t}(\dot{B}^{\frac{d}{p}+1}_{p,1})}\lesssim (1+\mathcal{X}^2_p(t))\mathcal{X}^2_p(t)
\end{align*}
and
\begin{eqnarray*}
\|g_{3}\|^\ell_{{L}^{1}_{t}(\dot{B}^{\frac{d}{p}-1}_{p,1})}
\lesssim\|G(a)\|_{\widetilde{L}^{\infty}_{t}(\dot{B}^{\frac{d}{p}}_{p,1})}
\|\nabla a\|_{L^{1}_{t}(\dot{B}^{\frac{d}{p}-1}_{p,1})}\lesssim \|a\|_{\widetilde{L}^{\infty}_{t}(\dot{B}^{\frac{d}{p}}_{p,1})} \|a\|_{L^{1}_{t}(\dot{B}^{\frac{d}{p}-1}_{p,1})}\lesssim  \mathcal{X}^2_p(t).
\end{eqnarray*}

Finally, substituting the above nonlinear estimates into \eqref{high:es} and
\eqref{Lp-lowfreq}, we obtain \eqref{3.3} and complete the proof of
Proposition~\ref{prop31}.
\end{proof}

\vspace{2mm}
\subsection{Proof of global well-posedness}

Before proving global existence, let us state the local well-posedness result in a $L^p$ framework.

\begin{thm}\label{thm3.1}Let $d\geq2$ and $ p\in[1,2d)$. Suppose that $(\rho_0,u_0)$ with $\rho_0=1+a_0$ satisfying
$$
\inf\limits_{x\in\mathbb{R}^d}\rho_0(x)>0,\quad \quad  a_0 \in\dot{B}_{p,1}^{\frac dp-2}\cap\dot{B}_{p,1}^{\frac dp} \quad\text{and}\quad u_{0}\in \dot{B}_{p,1}^{\frac dp-1}.
$$
In the case $p=1$, additionally assume that
$$
\|a_0\|_{\dot{B}_{p,1}^{\frac dp}}\leq \varepsilon\ll1.
$$
Then,  there exists a time $T>0$ such that the Cauchy problem \eqref{1.1}-\eqref{1.2} admits a unique solution $(\rho,u)$ with $\rho=1+a$ on $[0,T]\times\mathbb{R}^d$, which satisfies
$$
\inf_{(t,x)\in [0,T]\times \mathbb{R}^d}\rho(t,x)>0,\quad a\in\mathcal{C}([0,T];\dot{B}_{p,1}^{\frac dp-2}\cap\dot{B}_{p,1}^{\frac dp}),$$
and
$$u\in\mathcal{C}([0,T];\dot{B}_{p,1}^{\frac dp-1})\cap  L^1(0,T;\dot{B}_{p,1}^{\frac dp+1}).
$$

\end{thm}

We postpone the proof of local well-posedness in the Appendix.

\vspace{2mm}
\noindent
\textbf{Proof of Theorem \ref{thm2.1}.} Under the smallness assumption \eqref{a1} (with $\varepsilon_0\leq \varepsilon$), Theorem~\ref{thm3.1} ensures that the existence and uniqueness of a solution $(\rho,u)$ with $\rho=1+a$ on $[0,T]\times\mathbb{R}^d$ hold for a small time $T$ such that
\[
\mathcal{X}_{p}(T)\leq 2C^*\mathcal{X}_{p}(0).
\]
By the embedding $\dot{B}^{\frac{d}{p}}_{p,1}\hookrightarrow L^{\infty}$, the condition \eqref{3.2} is fulfilled provided that $\mathcal{X}_{p}(0)\leq \varepsilon_0\leq \varepsilon^*$ for some constant $\varepsilon^*>0$. Therefore, we may apply the {\emph{a priori}} estimate in Proposition~\ref{prop31} to deduce that
\begin{align*}
\mathcal{X}_{p}(T)
&\leq C^*\mathcal{X}_{p}(0)
+C^*\bigl(1+2C^*\mathcal{X}_{p}(0)\bigr)\,4(C^*)^2\mathcal{X}_{p}(0)^2
\leq \frac{3}{2}C^*\,\mathcal{X}_{p}(0),
\end{align*}
as long as
\[
\mathcal{X}_{p}(0)\leq \varepsilon_0\leq
\min\left\{\varepsilon^*,\,\frac{1}{2C^*},\,\frac{1}{16(C^*)^2}\right\}.
\]
Hence, applying the local well-posedness result again with initial data
$(\rho(T),u(T))$, we can extend the solution to $[0,T+T_1]$ for some $T_1>0$. Repeating the above procedure
step by step shows that the solution can be prolonged globally in time. The proof of Theorem \ref{thm2.1} is thus completed.


\section{Asymptotic behavior: Proof of Theorem \ref{Theorem2.2}}

In this section, we investigate the optimal decay estimates of the fluctuation of $L^p$ solutions.

\subsection{Uniform evolution of the low-frequency norm}

The decay rates rely on the following Proposition, which concerns a uniform propagation of the low-frequency
$\dot{B}^{\sigma_1}_{p,\infty}$ norm:

\begin{prop}\label{lemma:evolution:gamma>0}

Let $-1-\min\{\frac{d}{p},\frac{d}{p'}\}\leq \sigma_1<\frac{d}{p}-1$.  If $(a,m)$ with $m=(1+a)u$ is the solution to \eqref{re:momentum} satisfying $\mathcal{X}_p(t)<\infty$, then we have
\begin{equation}\label{4.55}
\begin{aligned}
\|(\Lambda^{-1} a,m)^{\ell}\|_{L^{\infty}_t(\dot{B}^{\sigma_1}_{p,\infty})}\lesssim e^{(1+\mathcal{X}_{p}(t))^2\mathcal{X}_{p}(t)}\Big(\|(\Lambda^{-1} a_0, m_0)^{\ell}\|_{\dot{B}^{\sigma_1}_{p,\infty}}+(1+\mathcal{X}_{p}(t))^5\mathcal{X}_{p}(t)\Big).
\end{aligned}
\end{equation}
\end{prop}

\begin{proof}
Applying the low-frequency cutoff to \eqref{re:momentum} and using Lemma~\ref{heat Lp}, we obtain 
\begin{equation}\label{4.3}
\begin{aligned}
\|(\Lambda^{-1}a,m)^{\ell}\|_{L^{\infty}_t(\dot{B}^{\sigma_1}_{p,\infty})\cap \widetilde{L}^1_t(\dot{B}^{\sigma_1+2}_{p,\infty})}
&\lesssim \|(\Lambda^{-1} a_0, m_0)^{\ell}\|_{\dot{B}^{\sigma_1}_{p,\infty}}+\|\dive \mathcal{N}^{\ell}\|_{\widetilde{L}^1_t(\dot{B}^{\sigma_1}_{p,\infty})}.
\end{aligned}
\end{equation}
To proceed, we analyze the nonlinear term $\mathcal{N}$. Before that, we claim that, for all $\min\{\frac{d}{p},\frac{d}{p'}\}\leq s\leq \frac{d}{p}$,
\begin{equation}\label{mmndd}
\begin{aligned}
\|u^{\ell}\|_{\dot{B}^{s}_{p,\infty}}\lesssim (1+\|a\|_{\dot{B}^{\frac{d}{p}}_{p,1}}) \|m^\ell\|_{\dot{B}^{s}_{p,\infty}}+(1+\|a\|_{\dot{B}^{\frac{d}{p}}_{p,1}})\|u\|_{\dot{B}^{\frac{d}{p}-1}_{p,1}}\|a\|_{\dot{B}^{\frac{d}{p}}_{p,1}}+\|u\|_{\dot{B}^{s}_{p,1}}^h.
\end{aligned}
\end{equation}
Indeed, there holds
\begin{equation*}
\begin{aligned}\|u^{\ell}\|_{\dot{B}^{s}_{p,\infty}}&\lesssim\|m^\ell\|_{\dot{B}^{s}_{p,\infty}}+\|(I(a)m)^\ell\|_{\dot{B}^{s}_{p,\infty}}+\|u\|_{\dot{B}^{s}_{p,1}}^h.
\end{aligned}
\end{equation*}
Further, we have
\begin{equation*}
\begin{aligned}\|(I(a)m)^\ell\|_{\dot{B}^{s}_{p,\infty}}&\lesssim\|I(a)m^{\ell}\|_{\dot{B}^{s}_{p,\infty}}^\ell+\|I(a)m^h\|_{\dot{B}^{-\min\{\frac{d}{p},\frac{d}{p'}\}}_{p,\infty}}^\ell\\
&\lesssim \|I(a)\|_{\dot{B}^{\frac{d}{p}}_{p,1}}\|m^\ell\|_{\dot{B}^{s}_{p,\infty}}+\|I(a)\|_{\dot{B}^{\frac{d}{p}}_{p,1}}\|m^h\|_{\dot{B}^{-\min\{\frac{d}{p},\frac{d}{p'}\}}_{p,\infty}}\\
&\lesssim \|a\|_{\dot{B}^{\frac{d}{p}}_{p,1}}\big(\|m^\ell\|_{\dot{B}^{s}_{p,\infty}}+\|(1+a)u\|_{\dot{B}^{\frac{d}{p}-1}_{p,1}}^h\big)
\\
&\lesssim \|a\|_{\dot{B}^{\frac{d}{p}}_{p,1}} \|m^{\ell}\|_{\dot{B}^{s}_{p,\infty}}+\|a\|_{\dot{B}^{\frac{d}{p}}_{p,1}}(1+\|a\|_{\dot{B}^{\frac{d}{p}}_{p,1}})\|u\|_{\dot{B}^{\frac{d}{p}-1}_{p,1}}.
\end{aligned}
\end{equation*}
Then, regarding the first term involving $u\otimes u$, we consider it in two cases: 
\begin{itemize}

\item {\emph{Case 1: $-\min\{\frac{d}{p},\frac{d}{p'}\}\leq \sigma_1<\frac{d}{p}-1$.}} 

In this case, the usual product law in $\dot{B}^{\sigma_1}_{p,\infty}$ is admissible. So, one may write $\dive (u\otimes u)=u\cdot\nabla u+u\dive u$ and employ Proposition \ref{prop3.2} ($s_1=\sigma_1,  s_2=\frac{d}{p}$) and \eqref{mmndd} to deduce
\begin{equation*}
\begin{aligned}
\|\dive(u\otimes u)^{\ell}\|_{\widetilde{L}^1_t(\dot{B}^{\sigma_1}_{p,\infty})}&\lesssim \int_0^t \|\nabla u\|_{\dot{B}^{\frac{d}{p}}_{p,1}} \|u\|_{\dot{B}^{\sigma_1}_{p,\infty}}\,d\tau\\
&\lesssim (1+\|a\|_{\widetilde{L}^{\infty}_t(\dot{B}^{\frac{d}{p}}_{p,1})})\int_0^t \|u\|_{\dot{B}^{\frac{d}{p}+1}_{p,1}}\|m^{\ell}\|_{\dot{B}^{\sigma_1}_{p,\infty}}\,d\tau\\
&\quad+(1+\|a\|_{\widetilde{L}^{\infty}_t(\dot{B}^{\frac{d}{p}}_{p,1})})(1+\|u\|_{\widetilde{L}^{\infty}_t(\dot{B}^{\frac{d}{p}-1}_{p,1})})\|u\|_{L^1_t(\dot{B}^{\frac{d}{p}+1}_{p,1})}
\end{aligned}
\end{equation*}

\item {\emph{Case 2: $-\min\{\frac{d}{p},\frac{d}{p'}\}-1\leq \sigma_1< -\min\{\frac{d}{p},\frac{d}{p'}\}$.}} 

We shall take advantage of the property of "one-derivative gain". Since $-\min\{\frac{d}{p},\frac{d}{p'}\}-1< \sigma_1< -\min\{\frac{d}{p},\frac{d}{p'}\}\leq \frac{d}{p}-1$,  $\sigma_1^*=\frac{1}{2}(\sigma_1+1+\frac{d}{p})$ satisfies $2\sigma_1^*-\frac{d}{p}=\sigma_1+1$ and $-\min\{\frac{d}{p},\frac{d}{p'}\}< \sigma_1^*<\frac{d}{p}$.  It thus holds by the third product law in Proposition \ref{prop3.2} ($s_1=s_2=\sigma_1^*$) and \eqref{mmndd} that
\begin{equation*}
\begin{aligned}
\|\dive (u\otimes u)^{\ell}\|_{\widetilde{L}^1_t(\dot{B}^{\sigma_1}_{p,\infty})}&\lesssim \int_0^t \|u\otimes u\|_{\dot{B}^{\sigma_1+1}_{p,\infty}}\,d\tau\lesssim \int_0^t \|u\|_{\dot{B}^{\sigma_1^*}_{p,1}}^2\,d\tau\\
&\lesssim (1+\|a\|_{\widetilde{L}^{\infty}_t(\dot{B}^{\frac{d}{p}}_{p,1})})^2 \int_0^t\|m^\ell\|_{\dot{B}^{\sigma_1^*}_{p,1}}^2\,d\tau\\
&\quad+(1+\|a\|_{\widetilde{L}^{\infty}_t(\dot{B}^{\frac{d}{p}}_{p,1})})^2 \|u\|_{\widetilde{L}^{\infty}_t(\dot{B}^{\frac{d}{p}-1}_{p,1})}^2 \|a\|_{\widetilde{L}^2_t(\dot{B}^{\frac{d}{p}}_{p,1})}^2+\|u\|_{\widetilde{L}^2_t(\dot{B}^{\frac{d}{p}}_{p,1})}^2.
\end{aligned}
\end{equation*}
Note that the real interpolation ensures that
\begin{align*}
\int_0^t\|m^\ell\|_{\dot{B}^{\sigma_1^*}_{p,1}}^2\,d\tau&\lesssim  \int_0^t\|m^\ell\|_{\dot{B}^{\sigma_1}_{p,\infty}} \|m^\ell\|_{\dot{B}^{\frac{d}{p}+1}_{p,1}}\,d\tau\\
&\lesssim \int_0^t  \Big(\|u\|_{\dot{B}^{\frac{d}{p}+1}_{p,1}}+ \|u\|_{\dot{B}^{\frac{d}{p}-1}_{p,1}}\|a\|_{\dot{B}^{\frac{d}{p}}_{p,1}}\Big)\|m^\ell\|_{\dot{B}^{\sigma_1}_{p,\infty}}\,d\tau,
\end{align*}
while we have the control 
\begin{align*}
\|m^\ell\|_{\dot{B}^{\frac{d}{p}+1}_{p,1}}\lesssim \|u^\ell\|_{\dot{B}^{\frac{d}{p}+1}_{p,1}}+\|au\|_{\dot{B}^{\frac{d}{p}-1}_{p,1}}^\ell\lesssim \|u\|_{\dot{B}^{\frac{d}{p}+1}_{p,1}}+\|a\|_{\dot{B}^{\frac{d}{p}}_{p,1}} \|u\|_{\dot{B}^{\frac{d}{p}-1}_{p,1}},
\end{align*}
due to $m=u+au$.
\end{itemize}

Combining the above two cases, we deduce that
\begin{align*}
\|\dive (u\otimes u)^{\ell}\|_{\widetilde{L}^1_t(\dot{B}^{\sigma_1}_{p,\infty})}&\lesssim (1+\mathcal{X}_{p}(t))^2\int_0^t \|u\|_{\dot{B}^{\frac{d}{p}+1}_{p,1}} \|m^{\ell}\|_{\dot{B}^{\sigma_1}_{p,\infty}}\,d\tau+(1+\mathcal{X}_{p}(t))^5\mathcal{X}_{p}(t).
\end{align*}

Now for the viscous term, we only focus on interaction $\div(\tilde\mu_1(a)Dm)$
like the convection term, we have two situations to consider:
\begin{itemize}

\item {\emph{Case 1: $-\min\{\frac{d}{p},\frac{d}{p'}\}\leq \sigma_1<\frac{d}{p}-1$.}} 

In this case, we deduce
\begin{equation*}
\begin{aligned}
\|\dive (\tilde\mu_1(a) Dm)^{\ell}\|_{\widetilde{L}^1_t(\dot{B}^{\sigma_1}_{p,\infty})}&\lesssim\int^{t}_0\|\tilde\mu_1(a) Dm\|_{\dot{B}^{\sigma_1}_{p,\infty}}d\tau\\
&\lesssim\int^{t}_0\|a\|_{\dot{B}^{\frac{d}{p}}_{p,1}}\|Dm\|_{\dot{B}^{\sigma_1}_{p,\infty}}d\tau\\ &\lesssim\int^{t}_0\|a\|_{\dot{B}^{\frac{d}{p}}_{p,1}}\|m^{\ell}\|_{\dot{B}^{\sigma_1}_{p,\infty}}d\tau
+(1+\mathcal{X}_{p}(t))\mathcal{X}_{p}(t).
\end{aligned}
\end{equation*}

\item {\emph{Case 2: $-\min\{\frac{d}{p},\frac{d}{p'}\}-1\leq \sigma_1< -\min\{\frac{d}{p},\frac{d}{p'}\}$.}} 

In this case, we deduce
\begin{equation*}
\begin{aligned}
\|\dive (\tilde\mu_1(a) Dm)^{\ell}\|_{\widetilde{L}^1_t(\dot{B}^{\sigma_1}_{p,\infty})}&\lesssim\int^{t}_0\|\tilde\mu_1(a) Dm\|_{\dot{B}^{\sigma_1+1}_{p,\infty}}d\tau\\
&\lesssim\int^{t}_0\|a\|_{\dot{B}^{\frac{d}{p}}_{p,1}}\|Dm\|_{\dot{B}^{\sigma_1+1}_{p,\infty}}d\tau\\ &\lesssim\int^{t}_0\|a\|_{\dot{B}^{\frac{d}{p}}_{p,1}}\|m^{\ell}\|_{\dot{B}^{\sigma_1}_{p,\infty}}d\tau
+(1+\mathcal{X}_{p}(t))\mathcal{X}_{p}(t).
\end{aligned}
\end{equation*}
\end{itemize}
Therefore, we are able to obtain
\begin{equation*}
\begin{aligned}
\|\dive (\tilde\mu_1(a) Dm)^{\ell}\|_{\widetilde{L}^1_t(\dot{B}^{\sigma_1}_{p,\infty})}\lesssim\int^{t}_0\|a\|_{\dot{B}^{\frac{d}{p}}_{p,1}}\|m^{\ell}\|_{\dot{B}^{\sigma_1}_{p,\infty}}d\tau
+(1+\mathcal{X}_{p}(t))\mathcal{X}_{p}(t).
\end{aligned}
\end{equation*}
Other viscous terms share similar calculations.

Furthermore, regarding the Poisson term, we apply Proposition \ref{prop3.2} and interpolation to verify that
\begin{align*}
\||\nabla \Lambda^{-2} a|^2 \|_{\widetilde{L}^1_t(\dot{B}^{\sigma_1+1}_{p,\infty})}&\lesssim \int_0^t \|\nabla \Lambda^{-2} a\|_{\dot{B}^{\frac{d}{p}}_{p,1}} \|\nabla \Lambda^{-2} a\|_{\dot{B}^{\sigma_1+1}_{p,\infty}}\,d\tau\\
&\lesssim \int_0^t \|a\|_{\dot{B}^{\frac{d}{p}-1}_{p,1}} \|a\|_{\dot{B}^{\sigma_1}_{p,\infty}}\,d\tau\\
&\lesssim  \int_0^t \|a\|_{\dot{B}^{\frac{d}{p}}_{p,1}}^{\frac{1}{\frac{d}{p}+1-\sigma_1}} \|a\|_{\dot{B}^{\sigma_1-1}_{p,\infty}}^{\frac{\frac{d}{p}-\sigma_1}{\frac{d}{p}+1-\sigma_1}}  \|a\|_{\dot{B}^{\frac{d}{p}}_{p,1}}^{\frac{\frac{d}{p}-\sigma_1}{\frac{d}{p}+1-\sigma_1}} \|a\|_{\dot{B}^{\sigma_1-1}_{p,\infty}}^{\frac{1}{\frac{d}{p}+1-\sigma_1}} \,d\tau\\
&\lesssim  \int_0^t \|a\|_{\dot{B}^{\frac{d}{p}}_{p,1}} \|a^{\ell}\|_{\dot{B}^{\sigma_1-1}_{p,\infty}}\,d\tau+\|a\|_{\widetilde{L}^{\infty}_t(\dot{B}^{\frac{d}{p}}_{p,1})} \|a\|_{L^1_t(\dot{B}^{\frac{d}{p}}_{p,1})}^h.
\end{align*}
Here the last inequality we used the decomposition $a=a^\ell+a^h$. The term involving $\nabla \Lambda^{-2} a\otimes \nabla \Lambda^{-2} a$ can be addressed similarly. As for the pressure part, one may write $H(a)=a\widetilde{H}(a)$ with a smooth function $\widetilde{H}(a)$ satisfying $\widetilde{H}(0)=0$. Then, the usual product laws and continuity of the composition function $\widetilde{H}(a)$ imply
\begin{align*}
\|H(a)\|_{\widetilde{L}^{1}_t(\dot{B}^{\sigma_1+1}_{p,\infty})}\lesssim \int_0^t \|a\|_{\dot{B}^{\frac{d}{p}}_{p,1}} \|a\|_{\dot{B}^{\sigma_1+1}_{p,\infty}}\,d\tau\lesssim \int_0^t \|a\|_{\dot{B}^{\frac{d}{p}}_{p,1}} \|a^{\ell}\|_{\dot{B}^{\sigma_1-1}_{p,\infty}}\,d\tau+\|a\|_{\widetilde{L}^{\infty}_t(\dot{B}^{\frac{d}{p}}_{p,1})} \|a\|_{\widetilde{L}^{1}_t(\dot{B}^{\frac{d}{p}}_{p,1})}^h.
\end{align*}
Collecting the above nonlinear estimates into \eqref{4.3} leads to
\begin{align*}
&\quad\|(\Lambda^{-1}a,m)^{\ell}\|_{L^{\infty}_t(\dot{B}^{\sigma_1}_{p,\infty})\cap \widetilde{L}^1_t(\dot{B}^{\sigma_1+2}_{p,\infty})}\\
&\lesssim (1+\mathcal{X}_{p}(t))^2\int_0^t \|(\nabla u,a)\|_{\dot{B}^{\frac{d}{p}}_{p,1}} \|(\Lambda^{-1}a,m)^{\ell}\|_{\dot{B}^{\sigma_1}_{p,\infty}}\,d\tau+(1+\mathcal{X}_{p}(t))^5 \mathcal{X}_{p}(t),
\end{align*}
which, together with Gr\"onwall's inequality, gives rise to the desired bound \eqref{4.55}.
\end{proof}

\subsection{A time-weighted approach}
In this section, we provide a detailed proof of the upper bounds in the time-decay estimates.
To this end, we introduce the following time-weighted functional:
\begin{equation}\label{Dt}
\begin{aligned}
\mathcal{D}_{p,M}(t):&=\|\tau^M( \Lambda^{-1} a, u )^{\ell}\|_{\widetilde{L}^{\infty}_t(\dot{B}^{\frac{d}{p}+1}_{p,1})}+\|\tau^M( \Lambda^{-1} a,  u )^{\ell}\|_{L^{1}_t(\dot{B}^{\frac{d}{p}+3}_{p,1})}\\
&\quad+\|\tau^M( \Lambda a, u )\|_{\widetilde{L}^{\infty}_{t}(\dot{B}^{\frac{d}{p}-1}_{p,1})}^h+\|\tau^M(  a,\Lambda u )\|_{ L^{1}_t(\dot{B}^{\frac{d}{p}}_{p,1})}^h.
\end{aligned}
\end{equation}
We emphasize that $\mathcal{D}_{p,M}(t)$ is not equivalent to the functional 
$\mathcal{X}_p(t)$ (see \eqref{Xt}) used in global existence. Instead, it is designed to capture the
maximal dissipative regularity in the $L^\infty_t$–based framework,
which reveals an enhanced dissipative structure at low frequencies.
The weight exponent $M$ is introduced only to avoid a potential singularity as $t\to0$.
It does not affect the final decay rates since the weight is removed at the end of the argument
(see \eqref{4.32}). Furthermore, we keep the superscript $^\ell$ inside the norm, which will be convenient for applying interpolation inequalities.

 We have the following proposition.

\begin{prop}\label{lemma:low:weighted}
Let $(a,u)$ with $m=(1+a)u$ be the global solution to \eqref{linearized:1} such that $\mathcal{X}_{p}(t)<\infty$. For any $M\gg 1$ and $t>0$, if it holds that
\begin{equation}
\begin{aligned}
\mathcal{D}_{p,M}(t)&\lesssim e^{C\mathcal{X}_p(t)+C\mathcal{X}_{p}(t)^2}\Big(\mathcal{X}_{p}(t)+ \|(\Lambda^{-1} a,m)^{\ell}\|_{L^{\infty}_t(\dot{B}^{\sigma_1}_{p,\infty})} \Big)t^{M-\frac{1}{2}(\frac{d}{p}+1-\sigma_1)}.\label{4.32}
\end{aligned}
\end{equation}
Furthermore, if $X(t)<<1$, then we have
\begin{align}\label{highu}
\|\tau^M u\|_{\widetilde{L}^{\infty}_t(\dot{B}^{\frac{d}{p}+1}_{p,1})}^h\lesssim \mathcal{D}_{p,M}(t)+\mathcal{X}_p(t).
\end{align}
Here, $\mathcal{X}_{p}$ and $\mathcal{D}_{p,M}$ are defined by \eqref{Xt} and \eqref{Dt}, respectively.
\end{prop}

\begin{proof}


Motivated by the low-frequency regularity gain \eqref{4.55}, we adopt different strategies in low and high frequencies: in low frequencies, we work with the momentum formulation and perform time-weighted estimates, whereas in high frequencies we switch back to the velocity variable in order to avoid derivative loss in  nonlinear terms (see also \cite{BSXZ} for the use of similar variables in the Fourier semigroup setting).

\begin{itemize}

\item {\emph{Low-frequency analysis}}

\end{itemize}

    Multiplying \eqref{linearized:1} by $t^{M}$, we have
\begin{equation}\label{t:re:momentum}
\left\{
\begin{aligned}
&\partial_t (t^M a) + \mathrm{div}\,(t^M m) = f_M,\\[1mm]
&\partial_t (t^M m)
+ \nabla (t^M a)+\frac{\nabla}{-\Delta} (t^M a)
- \bar{\mathcal{A}}(t^M m)
= g_M,\\
&(t^M a, t^M m)(0,x)=(0,0)
\end{aligned}
\right.
\end{equation}
with
$$
f_M=  M t^{M-1}a,\quad g_M=Mt^{M-1} m+\dive t^M \mathcal{N}
$$
Here, the quadratic term $\mathcal{N}$ is defined  in \eqref{N}. Then,  Duhamel's principle for \eqref{t:re:momentum} yields
\begin{equation}\label{Deltaj}
    \begin{aligned}
\left(
\begin{array}{cc}
\dot{\Delta}_j(t^M \Lambda^{-1}H a)\\
\dot{\Delta}_j(t^M m)\\
\end{array}
         \right)
 =\int_0^t\mathcal{G}(t-s) \left(
\begin{array}{cc}
\Lambda^{-1}H \dot{\Delta}_j f_M\\
 \dot{\Delta}_j g_M 
\end{array}
         \right).
    \end{aligned}
\end{equation}
with $H={\rm Id}-\Delta$ (see Subsection \ref{Sublinear}).

For all $j\leq j_0$, Proposition \ref{heat Lp} applied to \eqref{Deltaj} ensures that
\begin{equation*}
    \begin{aligned}
    &\|\dot{\Delta}_j(t^M \Lambda^{-1}a,  t^Mm )^{\ell}\|_{L^p}\lesssim \int_0^t e^{-r_0 2^{2j}(t-s)} \|\dot{\Delta}_j(\Lambda^{-1}f_M, g_M)^{\ell}\|_{L^p}\,ds.
    \end{aligned}
\end{equation*}
Consequently, by Young's inequality for convolutions, we have
\begin{equation}\label{4.17}
    \begin{aligned}
    &\quad \|\dot{\Delta}_j(\tau^M \Lambda^{-1} a,  \tau^M m )^{\ell}\|_{L^{\infty}_t(L^p)}+2^{2j}\|\dot{\Delta}_j(\tau^M \Lambda^{-1} a,  \tau^Mm )^{\ell}\|_{L^{1}_t(L^p)}\\
    &\lesssim \|(\Lambda^{-1}f_M, g_M)^{\ell}\|_{L^{1}_t(L^p)}.
    \end{aligned}
\end{equation}
Note that the above inequality holds for all $j\in\mathbb{Z}$ since $\dot{\Delta}_j\dot{S}_{j_0}=0$ for $j\geq j_0+1$. Then, summing \eqref{4.17} over $j\in\mathbb{Z}$ with the weight $2^{(\frac{d}{p}+1)j}$ yields
\begin{equation*}
    \begin{aligned}
&\|\tau^M( \Lambda^{-1} a, m )^{\ell}\|_{\widetilde{L}^{\infty}_t(\dot{B}^{\frac{d}{p}+1}_{p,1})}+\|\tau^M( \Lambda^{-1} a,  m )^{\ell}\|_{L^{1}_t(\dot{B}^{\frac{d}{p}+3}_{p,1})}\lesssim \|(\Lambda^{-1}f_M, g_M)^{\ell}\|_{L^{1}_t(\dot{B}^{\frac{d}{p}+1}_{p,1})}
    \end{aligned}
\end{equation*}
Now we analyze the term
\begin{equation}\label{pppm}
    \begin{aligned}
    &\|(\Lambda^{-1}f_M, g_M)^{\ell}\|_{L^{1}_t(\dot{B}^{\frac{d}{p}+1}_{p,1})}\lesssim \|\tau^{M-1}(\Lambda^{-1} a,m)^{\ell}\|_{L^{1}_t(\dot{B}^{\frac{d}{p}+1}_{p,1})}+\|\tau^M \mathcal{N}^{\ell}\|_{L^{1}_t(\dot{B}^{\frac{d}{p}+2}_{p,1})}.
    \end{aligned}
\end{equation}
We now analyze the terms on the r.h.s. of \eqref{pppm} as follows. Let $\theta\in(0,1)$ be given by
$$
\sigma_1 \theta+(\frac{d}{p}+3)(1-\theta)=1.
$$
From the real interpolation, we obtain
\begin{equation*}
    \begin{aligned}
    &\quad \|\tau^{M-1}(\Lambda^{-1} a,m)^{\ell}\|_{L^{1}_t(\dot{B}^{\frac{d}{p}+1}_{p,1})}\\
    &=\int_0^t \tau^{M-1}\|(\Lambda^{-1} a,m)^{\ell}\|_{\dot{B}^{\sigma_1}_{p,\infty}}^{\theta} \|(\Lambda^{-1} a,m)^{\ell}\|_{\dot{B}^{\frac{d}{p}+3}_{p,1}}^{1-\theta}\,d\tau\\
    &\leq \|(\Lambda^{-1} a,m)^{\ell}\|_{L^{\infty}_t(\dot{B}^{\sigma_1}_{p,\infty})}^{\theta} \|(\Lambda^{-1} a,m)^{\ell}\|_{L^1_t(\dot{B}^{\frac{d}{p}+3}_{p,1})}^{1-\theta} \Big(\int_0^t \tau^{M-\frac{1}{1-\theta}}d\tau \Big)^{1-\theta}\\
    &\leq \eta \|(\Lambda^{-1} a,m)^{\ell}\|_{L^1_t(\dot{B}^{\frac{d}{p}+3}_{p,1})}+\eta^{-1}  \|(\Lambda^{-1} a,m)^{\ell}\|_{L^{\infty}_t(\dot{B}^{\sigma_1}_{p,\infty})} t^{M-\frac{1}{2}(\frac{d}{p}+1-\sigma_1)}.
    \end{aligned}
\end{equation*}

We now analyze every nonlinear term in $\mathcal{N}$.  Note that $u\otimes u$ can be decomposed by
\begin{equation*}
    \begin{aligned}
    (u\otimes u)^{\ell}=(u^{\ell}\otimes u^{\ell})^{\ell}+(u^{\ell}\otimes u^{h})^{\ell}+(u^{h}\otimes u^{\ell})^{\ell}+(u^{h}\otimes u^{h})^{\ell}.
    \end{aligned}
\end{equation*}
Thus, Moser type product law (the first inequality in Proposition \ref{prop3.2}) and interpolation properties lead to
\begin{equation*}
    \begin{aligned}
  & \quad\| \tau^{M}(u^{\ell}\otimes u^{\ell})^{\ell}\|_{L^1_t(\dot{B}^{\frac{d}{p}+2}_{p,1})}\\
  &\lesssim \int_0^t \tau^{M}\|u^\ell\|_{\dot{B}^{\frac{d}{p}}_{p,1}}\| u^\ell \|_{\dot{B}^{\frac{d}{p}+2}_{p,1}}\,d\tau\\
  &\lesssim \int_0^t\tau^{M} \|u^\ell\|_{\dot{B}^{\frac{d}{p}}_{p,1}} \Big(\tau^{M}\| u^\ell \|_{\dot{B}^{\frac{d}{p}+1}_{p,1}}\Big)^{\frac{1}{2}}\Big(\tau^{M}\| u^\ell \|_{\dot{B}^{\frac{d}{p}+3}_{p,1}}\Big)^{\frac{1}{2}}\,d\tau \\
  &\lesssim \eta \|\tau^{M} u^\ell \|_{L^1_t(\dot{B}^{\frac{d}{p}+3}_{p,1})}+\eta^{-1} \int_0^t\|u^\ell\|_{\dot{B}^{\frac{d}{p}}_{p,1}}^2 \tau^{M}\| u^\ell \|_{\dot{B}^{\frac{d}{p}+1}_{p,1}}\,d\tau.
    \end{aligned}
\end{equation*}
On the other hand, one also has
\begin{equation*}
    \begin{aligned}
  & \quad \|\tau^{M} (u^{\ell}\otimes u^{h},u^{h}\otimes u^{\ell},u^h\otimes u^h)^{\ell}\|_{L^{1}_t(\dot{B}^{\frac{d}{p}+2}_{p,1})}\\
  &\lesssim  \|\tau^{M} (u^{\ell}\otimes u^{h},u^{h}\otimes u^{\ell},u^h\otimes u^h)\|_{L^{1}_t(\dot{B}^{\frac{d}{p}}_{p,1})}^{\ell}\\
  &\lesssim \int_0^t \tau^M \|u\|_{\dot{B}^{\frac{d}{p}}_{p,1}} \|u^h\|_{\dot{B}^{\frac{d}{p}}_{p,1}}\,d\tau\\
  &\lesssim \eta \|\tau^M u\|_{L^1_t(\dot{B}^{\frac{d}{p}+1}_{p,1})}^h+\eta^{-1}\int_0^t \|u\|_{\dot{B}^{\frac{d}{p}}_{p,1}}^2 \tau^M  \|u\|_{\dot{B}^{\frac{d}{p}-1}_{p,1}}^h\,d\tau.
    \end{aligned}
\end{equation*}
Since $H(0)=H'(0)=0$, one can write $H(a)=\widetilde{H}(a)a$ with a smooth function satisfying $\widetilde{H}(0)=0$. Using frequency cut-off and standard bounds for product and composite functions, we have
\begin{align*}
\|\tau^M H(a)\|_{L^{1}_t(\dot{B}^{\frac{d}{p}+3}_{p,1})}^{\ell}&\lesssim \int_0^t \tau^M \|\widetilde H(a)a\|_{\dot{B}^{\frac{d}{p}}_{p,1}}\,d\tau\lesssim \int_0^t \|a\|_{\dot{B}^{\frac{d}{p}}_{p,1}} \tau^M  \|a\|_{\dot{B}^{\frac{d}{p}}_{p,1}}\,d\tau.
\end{align*}
Concerning the nonlinearities from the Poisson term, one also finds that
\begin{align*}
\||\nabla\Lambda^{-2}a|^2\|_{L^{1}_t(\dot{B}^{\frac{d}{p}+2}_{p,1})}^{\ell}&\lesssim \int_0^t \tau^{M}\Big(\||\nabla\Lambda^{-2}a^{\ell}|^2\|_{\dot{B}^{\frac{d}{p}+2}_{p,1}}^{\ell}\\
&\quad+\|(\nabla\Lambda^{-2}a^{\ell}\cdot \nabla\Lambda^{-2}a^{h},\nabla\Lambda^{-2}a^{h}\cdot \nabla\Lambda^{-2}a^{\ell},\nabla\Lambda^{-2}a^{h}\cdot \nabla\Lambda^{-2}a^{h}) \|_{\dot{B}^{\frac{d}{p}}_{p,1}}^{\ell}\Big)\,d\tau\\
&\lesssim \int_0^t \tau^M \|\nabla\Lambda^{-2}a\|_{\dot{B}^{\frac{d}{p}}_{p,1}} \Big(\|\nabla\Lambda^{-2}a^{\ell}\|_{\dot{B}^{\frac{d}{p}+2}}+\|\nabla\Lambda^{-2}a^h\|_{\dot{B}^{\frac{d}{p}}_{p,1}}\Big)\,d\tau\\
&\lesssim \int_0^t \tau^M \|a\|_{\dot{B}^{\frac{d}{p}-1}_{p,1}}\Big( \|a^\ell\|_{\dot{B}^{\frac{d}{p}}_{p,1}}^{\frac{1}{2}} \|a^\ell\|_{\dot{B}^{\frac{d}{p}+2}_{p,1}}^{\frac{1}{2}}+\|a^h\|_{\dot{B}^{\frac{d}{p}}_{p,1}}^{\frac{1}{2}} \|a^h\|_{\dot{B}^{\frac{d}{p}}_{p,1}}^{\frac{1}{2}}\Big)\,d\tau\\
&\lesssim \eta \Big( \|\tau^M a^\ell\|_{L^1_t(\dot{B}^{\frac{d}{p}+2}_{p,1})}+\|\tau^M a\|_{L^1_t(\dot{B}^{\frac{d}{p}}_{p,1})}^h\Big)+\eta^{-1}\int_0^t \|a\|_{\dot{B}^{\frac{d}{p}-1}_{p,1}}^2\tau^M \|a\|_{\dot{B}^{\frac{d}{p}}_{p,1}}\,d\tau,
\end{align*}
and similar bound is true for $\nabla \Lambda^{-2} a\otimes \nabla \Lambda^{-2} a$. 

Concerning the terms arising from viscosities, one may write 
\begin{align*}
2\tilde\mu_1(a)Dm+2(\bar\mu_1+\tilde\mu_1(a))D(Q(a)m)&=2\tilde\mu_1(a)Du-2\bar\mu_1 D(au).
\end{align*}
Similar calculations yield
\begin{equation*}
    \begin{aligned}
 \|\tau^M \widetilde{\mu}_1(a) D u\|_{L^{1}_t(\dot{B}^{\frac{d}{p}+1}_{p,1})}^{\ell} &\lesssim \int_0^t \tau^M\|\widetilde{\mu}_1(a) D u\|_{\dot{B}^{\frac{d}{p}}_{p,1}}^{\ell}\,d\tau\lesssim \int_0^t\|u\|_{\dot{B}^{\frac{d}{p}+1}_{p,1}}  \tau^M \|a\|_{\dot{B}^{\frac{d}{p}}_{p,1}} \,d\tau,
    \end{aligned}
\end{equation*}
and
\begin{equation*}
    \begin{aligned}
&\quad \|\tau^M D(au)\|_{L^{1}_t(\dot{B}^{\frac{d}{p}+2}_{p,1})}^{\ell} \\
&\lesssim \int_0^t\tau^M\Big( \|a^\ell u^\ell\|_{\dot{B}^{\frac{d}{p}+3}_{p,1}}^\ell+  \|(a^\ell u^h,a^h u^\ell, a^h u^h)\|_{\dot{B}^{\frac{d}{p}}_{p,1}}^\ell\Big)\,d\tau\\
&\lesssim \int_0^t\tau^M  \|(a,u)\|_{\dot{B}^{\frac{d}{p}}}\Big( \|(a,u)^\ell\|_{\dot{B}^{\frac{d}{p}+3}_{p,1}}+\|(a,u)^h\|_{\dot{B}^{\frac{d}{p}}_{p,1}}\Big)\,d\tau\\
&\lesssim \eta \Big(\|\tau^M (a,u)^\ell\|_{L^1_t(\dot{B}^{\frac{d}{p}+3}_{p,1})}+\|\tau^M (a,u)\|_{L^1_t(\dot{B}^{\frac{d}{p}}_{p,1})}^h\Big)+\eta^{-1} \int_0^t \|(a,u)\|_{\dot{B}^{\frac{d}{p}}_{p,1}}^2 \tau^M \|(\Lambda^{-1}a,u)^\ell\|_{\dot{B}^{\frac{d}{p}+1}_{p,1}}\,d\tau.
    \end{aligned}
\end{equation*}
Collecting the above nonlinear estimates, we have
\begin{equation}
    \begin{aligned}
&\quad\|\tau^M( \Lambda^{-1} a, m )^{\ell}\|_{\widetilde{L}^{\infty}_t(\dot{B}^{\frac{d}{p}+1}_{p,1})}+\|\tau^M( \Lambda^{-1} a,  m )^{\ell}\|_{L^{1}_t(\dot{B}^{\frac{d}{p}+3}_{p,1})}\\
&\lesssim \eta\Big(\|\tau^M (a,u)^\ell\|_{L^1_t(\dot{B}^{\frac{d}{p}+3}_{p,1})}+\|\tau^M (a,u)\|_{L^1_t(\dot{B}^{\frac{d}{p}}_{p,1})}^h\Big)\\
&\quad\quad+\eta^{-1}\|(\Lambda^{-1}a,m)^\ell\|_{L^{\infty}_t(\dot{B}^{\sigma_1}_{p,\infty})} t^{M-\frac{1}{2}(\frac{d}{p}+1-\sigma_1)}\\
&\quad\quad+\eta^{-1} \int_0^t \|(a,u)\|_{\dot{B}^{\frac{d}{p}}_{p,1}}^2 \tau^M \Big(\|a^\ell\|_{\dot{B}^{\frac{d}{p}}_{p,1}}+\|u^\ell\|_{\dot{B}^{\frac{d}{p}+1}_{p,1}}\Big)\,d\tau+\int_0^t \|a\|_{\dot{B}^{\frac{d}{p}}_{p,1}} \tau^M  \|a\|_{\dot{B}^{\frac{d}{p}}_{p,1}}\,d\tau.\label{lowM0}
\end{aligned}
\end{equation}

To proceed, we need to explain the relation between the bounds for $u$ and $m$, which is stated as follows.
\begin{lem}\label{lemma42}
It holds that
\begin{equation}\label{um1}
\begin{aligned}
\|\tau^M u^\ell\|_{\widetilde{L}^{\infty}_t(\dot{B}^{\frac{d}{p}+1}_{p,1})}&\lesssim \|\tau^M m^\ell\|_{\widetilde{L}^{\infty}_t(\dot{B}^{\frac{d}{p}+1}_{p,1})} \\
&\quad+\mathcal{X}_{p}(t)\Big(\| \tau^M (\Lambda^{-1}a, u^\ell)\|_{\widetilde{L}^{\infty}_t(\dot{B}^{\frac{d}{p}+1}_{p,1})}+\| \tau^M (\Lambda a, u)\|_{\widetilde{L}^{\infty}_t(\dot{B}^{\frac{d}{p}-1}_{p,1})}^h\Big),
\end{aligned}
\end{equation}
and
\begin{equation}\label{um2}
\begin{aligned}
\|\tau^M u^\ell\|_{L^{1}_t(\dot{B}^{\frac{d}{p}+3}_{p,1})}&\lesssim \|\tau^M m^\ell\|_{L^{1}_t(\dot{B}^{\frac{d}{p}+3}_{p,1})}+\int_0^t \|(a,u)\|_{\dot{B}^{\frac{d}{p}}_{p,1}}^2 \tau^M \|(\Lambda^{-1}a,u)^\ell\|_{\dot{B}^{\frac{d}{p}+1}_{p,1}}\,d\tau.
\end{aligned}
\end{equation}
\end{lem}

\begin{proof}
As $u=m-au$, one may have
\begin{align*}
\|\tau^M u^\ell\|_{\widetilde{L}^{\infty}_t(\dot{B}^{\frac{d}{p}+1}_{p,1})}&\lesssim \|\tau^M m^\ell\|_{\widetilde{L}^{\infty}_t(\dot{B}^{\frac{d}{p}+1}_{p,1})}+\|\tau^M (a u)^\ell\|_{\widetilde{L}^{\infty}_t(\dot{B}^{\frac{d}{p}+1}_{p,1})}.
\end{align*}
Here, arguing similarly as before leads to
\begin{align*}
\|\tau^M (a u)^\ell\|_{\widetilde{L}^{\infty}_t(\dot{B}^{\frac{d}{p}+1}_{p,1})}&\lesssim \|\tau^M (a^\ell u^\ell)\|_{\widetilde{L}^{\infty}_t(\dot{B}^{\frac{d}{p}+1}_{p,1})}^\ell+\|\tau^M (a^\ell u^h,a^h u^\ell, a^h u^h)\|_{\widetilde{L}^{\infty}_t(\dot{B}^{\frac{d}{p}-1}_{p,1})}^\ell\\
&\lesssim \|a^\ell\|_{\widetilde{L}^{\infty}_t(\dot{B}^{\frac{d}{p}}_{p,1})} \| \tau^M u^\ell\|_{\widetilde{L}^{\infty}_t(\dot{B}^{\frac{d}{p}+1}_{p,1})}+\|u^\ell\|_{\widetilde{L}^{\infty}_t(\dot{B}^{\frac{d}{p}}_{p,1})} \| \tau^M a^\ell\|_{\widetilde{L}^{\infty}_t(\dot{B}^{\frac{d}{p}+1}_{p,1})}\\
&\quad+\|(\Lambda a,u)\|_{\widetilde{L}^{\infty}_t(\dot{B}^{\frac{d}{p}-1}_{p,1})} \|\tau^M (\Lambda a,u)^h\|_{\widetilde{L}^{\infty}_t(\dot{B}^{\frac{d}{p}}_{p,1})}\\
&\lesssim \|(\Lambda a,u)\|_{\widetilde{L}^{\infty}_t(\dot{B}^{\frac{d}{p}-1}_{p,1})} \Big(\| \tau^M (\Lambda a, u)^\ell\|_{\widetilde{L}^{\infty}_t(\dot{B}^{\frac{d}{p}-1}_{p,1})} +\| \tau^M (\Lambda a,  u)\|_{\widetilde{L}^{\infty}_t(\dot{B}^{\frac{d}{p}-1}_{p,1})}^h\Big).
\end{align*}
So, \eqref{um1} is proved.

Similarly, we have
\begin{align*}
\|\tau^M u^\ell\|_{L^{1}_t(\dot{B}^{\frac{d}{p}+3}_{p,1})}&\lesssim \|\tau^M m^\ell\|_{L^{1}_t(\dot{B}^{\frac{d}{p}+3}_{p,1})}+\|\tau^M (a u)^\ell\|_{L^{1}_t(\dot{B}^{\frac{d}{p}+3}_{p,1})}.
\end{align*}
The previous analysis shows that
\begin{align*}
\|\tau^M (a u)^\ell\|_{L^{1}_t(\dot{B}^{\frac{d}{p}+3}_{p,1})}&\lesssim \eta_1\|\tau^M (a,u)^\ell\|_{L^1_t(\dot{B}^{\frac{d}{p}+3}_{p,1})}+ \eta_1\int_0^t \|(a,u)\|_{\dot{B}^{\frac{d}{p}}_{p,1}}^2 \tau^M \|(\Lambda^{-1}a,u)^\ell\|_{\dot{B}^{\frac{d}{p}+1}_{p,1}}\,d\tau.
\end{align*}
Consequently, taking $\eta_1$ small enough, we get \eqref{um2}. This finishes the proof of Lemma \ref{lemma42}.\end{proof}


Then, recalling \eqref{lowM0} and using \eqref{um1}-\eqref{um2} to recover the bounds of $u$, we end up with 
\begin{equation}
    \begin{aligned}
&\quad\|\tau^M( \Lambda^{-1} a, m )^{\ell}\|_{\widetilde{L}^{\infty}_t(\dot{B}^{\frac{d}{p}+1}_{p,1})}+\|\tau^M( \Lambda^{-1} a,  m )^{\ell}\|_{L^{1}_t(\dot{B}^{\frac{d}{p}+3}_{p,1})}\\
&\lesssim \eta \Big(\|\tau^M (a,u)^\ell\|_{L^1_t(\dot{B}^{\frac{d}{p}+3}_{p,1})}+\|\tau^M (a,u)\|_{L^1_t(\dot{B}^{\frac{d}{p}}_{p,1})}^h\Big)\\
&\quad\quad+\eta^{-1}\|(\Lambda^{-1}a,m)^\ell\|_{L^{\infty}_t(\dot{B}^{\sigma_1}_{p,\infty})} t^{M-\frac{1}{2}(\frac{d}{p}+1-\sigma_1)}\\
&\quad\quad+\eta^{-1} \int_0^t \|(a,u)\|_{\dot{B}^{\frac{d}{p}}_{p,1}}^2 \tau^M \Big(\|a^\ell\|_{\dot{B}^{\frac{d}{p}}_{p,1}}+\|u^\ell\|_{\dot{B}^{\frac{d}{p}+1}_{p,1}}\Big)\,d\tau+\int_0^t \|a\|_{\dot{B}^{\frac{d}{p}}_{p,1}} \tau^M  \|a\|_{\dot{B}^{\frac{d}{p}}_{p,1}}\,d\tau.\label{lowM}
\end{aligned}
\end{equation}

\begin{itemize}

\item {\emph{High-frequency analysis}}

\end{itemize}

  Since there may be a loss of derivatives in high frequencies, we focus on the system \eqref{linearized:1} for $(a,u)$.  Multiplying \eqref{linearized:1} by $t^{M}$, we have
  \begin{equation}
\left\{
\begin{array}{l}\partial_{t}(t^M a)+u\cdot\nabla (t^M a)+(1+a)\mathrm{div}(t^M u)=\tilde{f}_M,\\ [1mm]
 \partial_{t}(t^M u)-\bar{\mathcal{A}}(t^M u)+\lambda\nabla (t^M a)+\frac{\nabla}{-\Delta} (t^M a)= \tilde{g}_M,\\[1mm]
(t^Ma,t^M u)(0,x)=(0,0).\\[1mm]
 \end{array} \right.\label{tMsys}
\end{equation}
with
\begin{equation*}
\begin{aligned}
\tilde{f}_M=Mt^{M-1} a,\quad \tilde{g}_M=M t^{M-1}u+t^M g.
\end{aligned}
\end{equation*}
Similarly to Step 1 in the proof of Proposition \ref{prop31}, one can obtain
\begin{equation}\label{427}
    \begin{aligned}
&\quad\|\tau^M( \nabla a, u )\|_{\widetilde{L}^{\infty}_t(\dot{B}^{\frac{d}{p}-1}_{p,1})}^h+\|\tau^M(  a, \nabla u )\|_{L^{1}_t(\dot{B}^{\frac{d}{p}}_{p,1})}^h\\
&\lesssim \|\tau^{M-1}( \nabla a, u )\|_{L^{1}_t(\dot{B}^{\frac{d}{p}-1}_{p,1})}^h+\|\tau^M g\|_{L^1_t(\dot{B}^{\frac{d}{p}-1}_{p,1})}^h\\
&\quad+\int_0^t \|\dive u\|_{\dot{B}^{\frac{d}{p}}_{p,1}}\tau^M\|a\|_{\dot{B}^{\frac{d}{p}}_{p,1}}^{h}\,d\tau
+\|\tau^M a u\|_{L^{1}_{t}(\dot{B}^{\frac{d}{p}}_{p,1})}^{h}+\|\tau^M (a\dive u)\|_{L^1_t(\dot{B}^{\frac{d}{p}-1}_{p,1})}^h.
\end{aligned}
\end{equation}
Here, one has
\begin{equation*}
    \begin{aligned}
    &\quad\|\tau^{M-1}( \nabla a, u )\|_{L^{1}_t(\dot{B}^{\frac{d}{p}-1}_{p,1})}^h\\
    & \lesssim \Big( \|(\nabla a,u)\|_{L^{\infty}_t(\dot{B}^{\frac{d}{p}}_{p,1})}^h
\Big)^{\theta} \Big(\|(\nabla a,u)\|_{L^1_t(\dot{B}^{\frac{d}{p}-1}_{p,1})}^h\Big)^{1-\theta} \Big(\int_0^t \tau^{M-\frac{1}{1-\theta}}d\tau \Big)^{1-\theta}\\
    & \leq \eta \|( a,\nabla u)\|_{L^1_t(\dot{B}^{\frac{d}{p}}_{p,1})}^h+\eta^{-1}  \|( a,\Lambda  u)\|_{L^{\infty}_t(\dot{B}^{\frac{d}{p}}_{p,1})}^h t^{M-\frac{1}{2}(\frac{d}{p}+1-\sigma_1)}.
\end{aligned}
\end{equation*}
Concerning $g$, we first deal with $g_1=-u\cdot\nabla u$. In fact, note that
$$
u\cdot\nabla u=u^\ell\cdot\nabla u^{\ell}+u^h\cdot\nabla u^\ell+u^\ell\cdot\nabla u^h+u^h\cdot\nabla u^h.
$$
Using the Moser-type product law, the frequency cut-off property, interpolation and Young's inequality, we have
\begin{equation*}
    \begin{aligned}
   &\quad \|\tau^{M}(u^\ell\cdot\nabla u^{\ell})\|_{L^1_t(\dot{B}^{\frac{d}{p}-1}_{p,1})}^h\\
   &\lesssim \int_0^t \|\tau^{M}(u^\ell\cdot\nabla u^{\ell})\|_{\dot{B}^{\frac{d}{p}+2}_{p,1}}^h\,d\tau\\
    & \lesssim \int_0^t \tau^M \Big(\|u^{\ell}\|_{\dot{B}^{\frac{d}{p}}_{p,1}}\|\nabla u^\ell\|_{\dot{B}^{\frac{d}{p}+2}_{p,1}}+\|\nabla u^{\ell}\|_{\dot{B}^{\frac{d}{p}}_{p,1}}\| u^\ell\|_{\dot{B}^{\frac{d}{p}+2}_{p,1}}\Big)\,d\tau\\
    & \lesssim \eta \|\tau^M u\|_{L^1_t(\dot{B}^{\frac{d}{p}+3}_{p,1})}^{\ell}+\eta^{-1}\int_0^t \|u\|_{\dot{B}^{\frac{d}{p}}_{p,1}}^2 \tau^M \|u^\ell\|_{\dot{B}^{\frac{d}{p}+1}_{p,1}}\,d\tau+\int_0^t \|u\|_{\dot{B}^{\frac{d}{p}+1}_{p,1}}\tau^M \|u^\ell\|_{\dot{B}^{\frac{d}{p}+1}_{p,1}}\,d\tau.
\end{aligned}
\end{equation*}
And concerning the terms involving at least one high-frequency component, it is easy to verify that
\begin{equation*}
   \quad \begin{aligned}
  &\quad \|\tau^{M}(u^h\cdot\nabla u^\ell, u^\ell\cdot\nabla u^h,u^h\cdot\nabla u^h)\|_{L^1_t(\dot{B}^{\frac{d}{p}-1}_{p,1})}^h\\
  &\lesssim \int_0^t \tau^M \|u^h\|_{\dot{B}^{\frac{d}{p}-1}_{p,1}} \|\nabla u\|_{\dot{B}^{\frac{d}{p}}_{p,1}}\,d\tau+\int_0^t \tau^M \|u\|_{\dot{B}^{\frac{d}{p}}_{p,1}} \|\nabla u^h\|_{\dot{B}^{\frac{d}{p}-1}_{p,1}}\,d\tau\\
  &\lesssim \eta \|\tau^M u\|_{L^1_t(\dot{B}^{\frac{d}{p}}_{p,1})}^h+\eta^{-1}\int_0^t \|u\|_{\dot{B}^{\frac{d}{p}}_{p,1}}^2 \tau^M \|u\|_{\dot{B}^{\frac{d}{p}-1}_{p,1}}^h \,d\tau.
\end{aligned}
\end{equation*}
One can easily have
\begin{equation*}
    \begin{aligned}
    &\quad \int_0^t \|\dive u\|_{\dot{B}^{\frac{d}{p}}_{p,1}}\tau^M\|a\|_{\dot{B}^{\frac{d}{p}}_{p,1}}^{h}\,d\tau
+\|\tau^M (a\dive u)\|_{L^1_t(\dot{B}^{\frac{d}{p}-1}_{p,1})}^h\\
    &\lesssim \int_0^t \|u\|_{\dot{B}^{\frac{d}{p}+1}_{p,1}} \tau^M \|a\|_{\dot{B}^{\frac{d}{p}}_{p,1}}\,d\tau
\end{aligned}
\end{equation*}
and
    \begin{equation*}
    \begin{aligned}
\|\tau^M a u\|_{L^{1}_{t}(\dot{B}^{\frac{d}{p}}_{p,1})}^{h}&\lesssim \int_0^t \tau^M \Big(\|a^\ell u^\ell\|_{\dot{B}^{\frac{d}{p}+2}_{p,1}}^h+ \|(a^h u^\ell,a^\ell u^h,a^h u^h)\|_{\dot{B}^{\frac{d}{p}}_{p,1}}^h\Big)\,d\tau\\
&\lesssim  \int_0^t \tau^M  \|(a,u)^\ell\|_{\dot{B}^{\frac{d}{p}}_{p,1}}  \|(a,u)^\ell\|_{\dot{B}^{\frac{d}{p}+2}_{p,1}}\,d\tau\\
&\quad+\int_0^t \tau^M \Big(\|a\|_{\dot{B}^{\frac{d}{p}}_{p,1}} \|u^h\|_{\dot{B}^{\frac{d}{p}}_{p,1}}+\|u\|_{\dot{B}^{\frac{d}{p}}_{p,1}}\|a^h\|_{\dot{B}^{\frac{d}{p}}_{p,1}}\Big)\,d\tau\\
&\lesssim \eta \Big(\|\tau^M (a,u)^\ell\|_{L^1_t(\dot{B}^{\frac{d}{p}+3}_{p,1})}+\|(a,\Lambda u)\|_{L^1_t(\dot{B}^{\frac{d}{p}}_{p,1})}^h\Big)\\
&\quad+\eta^{-1}\int_0^t \|(a,u)\|_{\dot{B}^{\frac{d}{p}}_{p,1}}^2 \tau^M \Big( \|(\Lambda^{-1}a,u)^\ell\|_{\dot{B}^{\frac{d}{p}+1}_{p,1}}+\|(\Lambda a, u)\|_{\dot{B}^{\frac{d}{p}-1}_{p,1}}^h\Big)\,d\tau.
\end{aligned}
\end{equation*}

Finally, to handle $g_2$, we only provide detailed estimates on $\dive (\widetilde{\mu}_1(a)D(u))$ and $\widetilde{Q}(a)\dive (\widetilde{\mu}_1(a)D(u))$. One has
    \begin{equation*}
    \begin{aligned}
    \|\tau^M\dive (\widetilde{\mu}_1(a)D(u))\|_{L^{1}_{t}(\dot{B}^{\frac{d}{p}-1}_{p,1})}^{h}&\lesssim \int_0^t \tau^M \|\widetilde{\mu}_1(a)\|_{\dot{B}^{\frac{d}{p}}_{p,1}} \|D(u)\|_{\dot{B}^{\frac{d}{p}}_{p,1}}\,d\tau\\
    &\lesssim \int_0^t \|u\|_{\dot{B}^{\frac{d}{p}+1}_{p,1}} \tau^M \|a\|_{\dot{B}^{\frac{d}{p}}_{p,1}}\,d\tau,
    \end{aligned}
    \end{equation*}
and
    \begin{equation*}
    \begin{aligned}
    \|\tau^M\widetilde{Q}(a)\dive (\widetilde{\mu}_1(a)D(u))\|_{L^{1}_{t}(\dot{B}^{\frac{d}{p}-1}_{p,1})}^{h}&\lesssim \int_0^t \tau^M \|\widetilde{Q}(a)\|_{\dot{B}^{\frac{d}{p}}_{p,1}} \|\dive (\widetilde{\mu}_1(a)D(u))\|_{\dot{B}^{\frac{d}{p}-1}_{p,1}}\,d\tau\\
    &\lesssim \int_0^t \|a\|_{\dot{B}^{\frac{d}{p}}_{p,1}} \|u\|_{\dot{B}^{\frac{d}{p}+1}_{p,1}} \tau^M \|a\|_{\dot{B}^{\frac{d}{p}}_{p,1}}   \,d\tau.
    \end{aligned}
    \end{equation*}

Therefore, collecting the above estimates into \eqref{427}, the high-frequency estimate reads:
\begin{equation}\label{highM}
    \begin{aligned}
&\quad\|\tau^M( \nabla a, u )\|_{\widetilde{L}^{\infty}_t(\dot{B}^{\frac{d}{p}-1}_{p,1})}^h+\|\tau^M(  a, \nabla u )\|_{L^{1}_t(\dot{B}^{\frac{d}{p}}_{p,1})}^h\\
&\lesssim \eta \Big(\|\tau^M (a,u)^\ell\|_{L^1_t(\dot{B}^{\frac{d}{p}+3}_{p,1})}+\|(a,\Lambda u)\|_{L^1_t(\dot{B}^{\frac{d}{p}}_{p,1})}^h\Big)+\eta^{-1}  \|( a,\Lambda  u)\|_{L^{\infty}_t(\dot{B}^{\frac{d}{p}}_{p,1})}^h t^{M-\frac{1}{2}(\frac{d}{p}+1-\sigma_1)}\\
&\quad+\eta^{-1}\int_0^t\Big( \|(a,u)\|_{\dot{B}^{\frac{d}{p}}_{p,1}}^2+(1+\|a\|_{\dot{B}^{\frac{d}{p}}_{p,1}})\|u\|_{\dot{B}^{\frac{d}{p}+1}_{p,1}}\Big) \\
&\quad\quad\times \tau^M \Big( \|(\Lambda^{-1}a,u)^\ell\|_{\dot{B}^{\frac{d}{p}+1}_{p,1}}+\|(\Lambda a, u)\|_{\dot{B}^{\frac{d}{p}-1}_{p,1}}^h\Big)\,d\tau.
\end{aligned}
\end{equation}

\begin{itemize}

\item {\emph{A G\"onwall-type argument}}

\end{itemize}

Collecting the above estimates \eqref{lowM} and \eqref{highM} and then choosing a sufficiently small constant $\eta>0$, we obtain
\begin{align*}
\mathcal{D}_{p,M}(t)&\lesssim \Big(\mathcal{X}_{p}(t)+ \|(\Lambda^{-1} a,m)^{\ell}\|_{L^{\infty}_t(\dot{B}^{\sigma_1}_{p,\infty})} \Big)t^{M-\frac{1}{2}(\frac{d}{p}+1-\sigma_1)}\\
&\quad+\int_0^t\Big( \|(a,u)\|_{\dot{B}^{\frac{d}{p}}_{p,1}}^2+(1+\|a\|_{\dot{B}^{\frac{d}{p}}_{p,1}})\|u\|_{\dot{B}^{\frac{d}{p}+1}_{p,1}}\Big)\mathcal{D}_{p,M}(\tau)\,d\tau.
\end{align*}
Taking advantage of Gr\"onwall's lemma and $\|(a,u)\|_{L^2_t(\dot{B}^{\frac{d}{p}}_{p,1})}\lesssim \mathcal{X}_{p}(t)$, we arrive at \eqref{4.32}.

Finally, using Lemma \ref{lemma6.1} ($\varrho_1=\infty$) to the second equation of \eqref{tMsys} and noting that $\tau^{M-1}\lesssim \tau^{M}+1$, we arrive at the higher-order estimate of $u$:
\begin{equation}\label{udp1}
\begin{aligned}
\|\tau^M u\|_{\widetilde{L}^{\infty}_{t}(\dot{B}^{\frac{d}{p}+1}_{p,1})}^h
&\lesssim \|\tau^M(\Lambda a, \Lambda^{-1}a)\|_{\widetilde{L}^{\infty}_{t}(\dot{B}^{\frac{d}{p}-1}_{p,1})}^h+\|\tilde{g}_{M}\|_{\widetilde{L}^{\infty}_{t}(\dot{B}^{\frac{d}{p}-1}_{p,1})}^h\\
&\lesssim   \|\tau^M (\Lambda a,u)\|_{\widetilde{L}^{\infty}_{t}(\dot{B}^{\frac{d}{p}-1}_{p,1})}^h+\|u\|_{\widetilde{L}^{\infty}_{t}(\dot{B}^{\frac{d}{p}-1}_{p,1})}^h\\
&\quad+(1+\|(\Lambda a,u)\|_{\widetilde{L}^{\infty}_t(\dot{B}^{\frac{d}{p}-1}_{p,1})})\|\tau^{M}(a,\Lambda u)\|_{\widetilde{L}^{\infty}_{t}(\dot{B}^{\frac{d}{p}}_{p,1})}\\
\end{aligned}
\end{equation}
Using \eqref{4.32} and $\mathcal{X}_p(t)<<1$, we end up with \eqref{highu} and finish the proof of Proposition \ref{lemma:low:weighted}.

\end{proof}

\subsection{Proof of Theorem \ref{Theorem2.2}}

Employing Proposition \ref{lemma:evolution:gamma>0} and \ref{lemma:low:weighted}, we obtain that, for any $M>>1$,
\begin{align}\label{4.28}
\mathcal{D}_{p,M}(t)\lesssim e^{(1+\mathcal{X}_{p}(t))^2\mathcal{X}_{p}(t)}\Big(\|(\Lambda^{-1} a_0, m_0)\|_{\dot{B}^{\sigma_1}_{p,\infty}}^\ell+(1+\mathcal{X}_{p}(t))^5\mathcal{X}_{p}(t)\Big) t^{-\frac{1}{2}(\frac{d}{p}+1-\sigma_1)},
\end{align}
We have already shown that $\mathcal{X}_p(t)<<1$ (more generally, the boundedness of $\mathcal{X}_p(t)$ is enough). Recalling the definition of $\mathcal{D}_{p,M}(t)$ and dividing both sides of \eqref{4.28} by $t^M$ and using \eqref{highu}, we have, for $t\geq1$,
\begin{equation}\label{decayhighest}
\begin{aligned}
\|a(t)\|_{\dot{B}^{\frac{d}{p}}_{p,1}}+\|u(t)\|_{\dot{B}^{\frac{d}{p}+1}_{p,1}}\lesssim t^{-\frac{1}{2}(\frac{d}{p}+1-\sigma_1)}.
\end{aligned}
\end{equation}
Combining real interpolation, \eqref{bound}, \eqref{4.55} and \eqref{decayhighest}, we infer, for $t\geq 1$ and $\sigma_1<\sigma\leq \frac{d}{p}+1$,
\begin{align}\label{adecay}
    &\|\Lambda^{-1} a(t)\|_{\dot{B}^{\sigma}_{p,1}}\lesssim \|\Lambda^{-1}a(t)\|_{\dot{B}^{\sigma_1}_{p,\infty}}^{\frac{\frac dp+1-\sigma}{\frac dp+1-\sigma_1}} \|\Lambda^{-1}a(t)\|_{\dot{B}^{\frac{d}{p}+1}_{p,1}}^{\frac{\sigma-\sigma_1}{\frac dp+1-\sigma_1}}\lesssim   (1+t)^{-\frac{1}{2}(\sigma-\sigma_1)},
    \end{align}
    which is exactly \eqref{decay1}.

    The decay of $u$ and $m$ shows a slight difference. Since $m=u+au$, the product law and regularity of $u$ imply that one may not expect the decay of $m$ with exponents higher than $\frac{d}{p}$, while if $\sigma_0-1\leq \sigma_1<\sigma_0$, $u$ will not exhibit decay in $\dot{B}^{\sigma}_{p,1}$ with $\sigma_1<\sigma\leq \sigma_0$.  By \eqref{bound}, \eqref{4.55}, \eqref{decayhighest}, one has
  \begin{align*}
    \|m^\ell(t)\|_{\dot{B}^{\frac{d}{p}+1}_{p,1}}&\lesssim \|u^\ell(t)\|_{\dot{B}^{\frac{d}{p}+1}_{p,1}}+\|au(t)\|_{\dot{B}^{\frac{d}{p}-1}_{p,1}}^{\ell}\\
    &\lesssim \|u^\ell(t)\|_{\dot{B}^{\frac{d}{p}+1}_{p,1}}+\|a(t)\|_{\dot{B}^{\frac{d}{p}}_{p,1}}\|u(t)\|_{\dot{B}^{\frac{d}{p}-1}_{p,1}}\lesssim (1+t)^{-\frac{1}{2}(\frac{d}{p}+1-\sigma_1)}
    \end{align*}
and
  \begin{align*}
    \|m^h(t)\|_{\dot{B}^{\frac{d}{p}}_{p,1}}&\lesssim \|u^h(t)\|_{\dot{B}^{\frac{d}{p}}_{p,1}}+\|au(t)\|_{\dot{B}^{\frac{d}{p}}_{p,1}}^{\ell}\\
    &\lesssim \|u(t)\|_{\dot{B}^{\frac{d}{p}+1}_{p,1}}^h+\|a(t)\|_{\dot{B}^{\frac{d}{p}}_{p,1}}\|u(t)\|_{\dot{B}^{\frac{d}{p}}_{p,1}}\lesssim (1+t)^{-\frac{1}{2}(\sigma-\sigma_1)}.
    \end{align*}
Here, we used $\|u(t)\|_{\dot{B}^{\frac{d}{p}}_{p,1}}\lesssim \|u(t)\|_{\dot{B}^{\frac{d}{p}-1}_{p,1}\cap \dot{B}^{\frac{d}{p}+1}_{p,1}}\lesssim 1$ for $t\geq1$. The real interpolation allows us to obtain
\begin{align*}
\|m^\ell(t)\|_{\dot{B}^{\sigma}_{p,1}}\lesssim \|m^\ell(t)\|_{\dot{B}^{\sigma_1}_{p,\infty}}^{\frac{\frac dp+1-\sigma}{\frac dp+1-\sigma_1}} \|m^\ell(t)\|_{\dot{B}^{\frac{d}{p}+1}_{p,1}}^{\frac{\sigma-\sigma_1}{\frac dp+1-\sigma_1}}\lesssim (1+t)^{-\frac{1}{2}(\sigma-\sigma_1)},\quad \sigma_1<\sigma\leq \frac{d}{p}+1.
\end{align*}
Consequently, we arrive at \eqref{decay4}.

Concerning the decay of $u$ in the case $\sigma_1\geq \sigma_0$, one may infer \eqref{decay3} according to $u=m+Q(a)u$, \eqref{decayhighest}, \eqref{decay4} and standard product laws. When $\sigma_0-1\leq \sigma_1<\sigma_0$,  it follows that
    \begin{align*}
    &\|u(t)\|_{\dot{B}^{\sigma_0}_{p,1}}\lesssim \|m(t)\|_{\dot{B}^{\sigma_0}_{p,1}}+\|a(t)\|_{\dot{B}^{\sigma_0}_{p,1}}\|u(t)\|_{\dot{B}^{\frac{d}{p}}_{p,1}}\lesssim (1+t)^{-\frac{1}{2}(\sigma_0-\sigma_1)}.
\end{align*}
This, combined with \eqref{decayhighest}, yields
\begin{align*}
\|u(t)\|_{\dot{B}^{\sigma}_{p,1}}\lesssim \|u(t)\|_{\dot{B}^{\sigma_0}_{p,1}}^{\frac{\frac dp+1-\sigma}{\frac dp+1-\sigma_0}} \|u(t)\|_{\dot{B}^{\frac{d}{p}+1}_{p,1}}^{\frac{\sigma-\sigma_0}{\frac dp+1-\sigma_0}}\lesssim (1+t)^{-\frac{1}{2}(\sigma-\sigma_1)},
\end{align*}
for $\sigma_0<\sigma\leq \frac{d}{p}+1$. We then get \eqref{decay3} and complete the proof of Theorem \ref{Theorem2.2}.

\section{Appendix}

\subsection{Littlewood-Paley theory}

For the reader's convenience, we briefly review the basic framework of Fourier localization and the Littlewood--Paley theory that will be used throughout the paper. Standard references include Chapters~2--3 of \cite{BCD}.


Let $\chi\in C_c^\infty(\mathbb{R}^d)$ be a radial function satisfying $0\le \chi\le1$ and $ \mathrm{supp}\,\chi\subset \{\xi:|\xi|\le 4/3\}$.
Define
\[
\varphi(\xi)=\chi(\xi/2)-\chi(\xi),
\]
so that $\varphi$ is supported in the annulus $\{\xi\in\mathbb{R}^d:3/4\le |\xi|\le 8/3\}$ and $\sum_{q\in\mathbb{Z}}\varphi(2^{-q}\xi)=1$ for $\xi\neq0$.

For any tempered distribution $f\in\mathcal{S}'$, we introduce the homogeneous dyadic blocks $\dot{\Delta}_j $ and  the low-frequency cut-off operators $\dot{S}_j$ by
\[
\dot{\Delta}_j f=\varphi(2^{-q}D)f\quad\text{and}\quad\dot{S}_j f=\chi(2^{-q}D)f \qquad q\in\mathbb{Z},
\]
where $\varphi(2^{-q}D)$ and $\chi(2^{-q}D)$ are defined as Fourier multipliers with symbols $\varphi(2^{-q}\xi)$ and $\chi(2^{-q}\xi)$, respectively.

Let $\mathcal{P}$ denote the space of polynomials and set $\mathcal{S}'_0=\mathcal{S}'/\mathcal{P}$. Then any $f\in\mathcal{S}'_0$ admits the homogeneous Littlewood–Paley decomposition
\[
f=\sum_{q\in\mathbb{Z}}\dot{\Delta}_j f \quad \text{in }\quad \mathcal{S}'_0 .
\]
Given the threshold $j_0$ (see Step 1 in Proposition \ref{prop31}), we further define the low- and high-frequency parts of $f$ as follows:
\[
f^\ell=\dot{S}_{j_0}f=\sum_{j\leq j_0-1}\dot{\Delta}_j f,\qquad 
f^h=({\rm Id}-\dot{S}_{j_0})f=\sum_{j\geq j_0}\dot{\Delta}_j f .
\]

\begin{defn}\label{defn2.1}
Let $s\in\mathbb{R}$ and $1\le p,r\le\infty$.  
The homogeneous Besov space $\dot{B}^s_{p,r}$ consists of all $f\in\mathcal{S}'_0$ such that
\[
\|f\|_{\dot{B}^s_{p,r}}
=\Big\|\Big\{2^{js}\|\dot{\Delta}_jf\|_{L^p}\Big\}_{j\in\mathbb{Z}}\Big\|_{l^{r}}.
\]
\end{defn}

We recall several classical properties of homogeneous Besov spaces (see \cite{BCD}):

\medskip
\noindent
$\bullet$ \emph{Scaling.}  
For any $\sigma\in\mathbb{R}$ and $1\le p,r\le\infty$, there exists $C>0$ such that for all $\lambda>0$,
\[
\|f(\lambda\,\cdot)\|_{\dot{B}^\sigma_{p,r}}
\approx \lambda^{\sigma-\frac{d}{p}}\|f\|_{\dot{B}^\sigma_{p,r}} .
\]

\medskip
\noindent
$\bullet$ \emph{Completeness.}  
The space $\dot{B}^\sigma_{p,r}$ is Banach whenever $\sigma<\frac{d}{p}$, or $\sigma=\frac{d}{p}$ and $r=1$.

\medskip
\noindent
$\bullet$ \emph{Interpolation.}  
Let $\sigma_1\neq\sigma_2$, $\theta\in(0,1)$, and $1\le p,r_1,r_2,r\le\infty$ with
\[
\frac1r=\frac{\theta}{r_1}+\frac{1-\theta}{r_2}.
\]
Then
\[
\|f\|_{\dot{B}^{\theta\sigma_1+(1-\theta)\sigma_2}_{p,r}}
\lesssim 
\|f\|_{\dot{B}^{\sigma_1}_{p,r_1}}^\theta
\|f\|_{\dot{B}^{\sigma_2}_{p,r_2}}^{1-\theta}.
\]

\medskip
\noindent
$\bullet$ \emph{Fourier multipliers.}  
If $F$ is a smooth homogeneous function of degree $m$ on $\mathbb{R}^d\setminus\{0\}$, then
\[
F(D):\dot{B}^\sigma_{p,r}\longrightarrow \dot{B}^{\sigma-m}_{p,r}.
\]


The following embedding properties will be used repeatedly throughout the paper.

\begin{prop}\label{Prop2.1}
The following statements hold:
\begin{itemize}
\item For any $p\in[1,\infty]$, we have the continuous embeddings
\[
\dot{B}^{0}_{p,1}\hookrightarrow L^{p}\hookrightarrow \dot{B}^{0}_{p,\infty}.
\]

\item Let $\sigma\in\mathbb{R}$, $1\leq p_{1}\leq p_{2}\leq\infty$ and $1\leq r_{1}\leq r_{2}\leq\infty$.
Then
\[
\dot{B}^{\sigma}_{p_1,r_1}\hookrightarrow
\dot{B}^{\sigma-d\left(\frac{1}{p_{1}}-\frac{1}{p_{2}}\right)}_{p_{2},r_{2}}.
\]

\item The space $\dot{B}^{\frac{d}{p}}_{p,1}$ is continuously embedded into the space of bounded
continuous functions, which additionally vanish at infinity if $p<\infty$.
\end{itemize}
\end{prop}

\medskip

We also recall the classical \emph{Bernstein inequality}:
\begin{equation}\label{Eq:2.6}
\|D^{k}f\|_{L^{b}}
\leq C^{1+k} \lambda^{k+d\left(\frac{1}{a}-\frac{1}{b}\right)}\|f\|_{L^{a}},
\end{equation}
which holds for all functions $f$ such that $\mathrm{Supp}\,\widehat{f}\subset\{\xi\in\mathbb{R}^{d}:|\xi|\leq R\lambda\}$ 
for some $R>0$ and $\lambda>0$, provided that $k\in\mathbb{N}$ and $1\leq a\leq b\leq\infty$.

More generally, if $\mathrm{Supp}\,\widehat{f}\subset\{\xi\in\mathbb{R}^{d}:R_{1}\lambda\leq|\xi|\leq R_{2}\lambda\}$ for some $0<R_{1}<R_{2}$ and $\lambda>0$, then for any smooth homogeneous function $A$ of degree $m$
on $\mathbb{R}^{d}\setminus\{0\}$ and any $1\leq a\leq\infty$, one has (see e.g. Lemma~2.2 in \cite{BCD})
\begin{equation}\label{Eq:2.7}
\|A(D)f\|_{L^{a}}\approx \lambda^{m}\|f\|_{L^{a}}.
\end{equation}
As a direct consequence of \eqref{Eq:2.6} and \eqref{Eq:2.7}, we have
\[
\|D^{k}f\|_{\dot{B}^{s}_{p,r}}\approx \|f\|_{\dot{B}^{s+k}_{p,r}},
\qquad k\in\mathbb{N}.
\]

\medskip

When studying evolution equations, we also make use of mixed space–time Besov spaces, introduced by
Chemin and Lerner~\cite{CL}.

\begin{defn}\label{defn2.2}
Let $T>0$, $s\in\mathbb{R}$ and $1\leq r,\theta\leq\infty$.
The homogeneous Chemin–Lerner space $\widetilde{L}^{\varrho}_{T}(\dot{B}^{s}_{p,r})$ is defined as
\[
\widetilde{L}^{\varrho}_{T}(\dot{B}^{s}_{p,r})
=\Big\{f\in L^{\theta}(0,T;\mathcal{S}'_{0}) :
\|f\|_{\widetilde{L}^{\varrho}_{T}(\dot{B}^{s}_{p,r})}<\infty\Big\},
\]
where
\[
\|f\|_{\widetilde{L}^{\varrho}_{T}(\dot{B}^{s}_{p,r})}
=
\begin{cases}
\displaystyle
\Big\|\Big\{2^{js}\|\dot{\Delta}_jf\|_{L^\theta(L^p)}\Big\}_{j\in\mathbb{Z}}\Big\|_{l^{r}}
& 1\le \varrho<\infty,\\[2ex]
\displaystyle
\Big\|\Big\{2^{js}\sup_{t\in[0,T]}\|\dot{\Delta}_jf\|_{L^p}\Big\}_{j\in\mathbb{Z}}\Big\|_{l^{r}}
& \varrho=\infty,\\[2ex]
\end{cases}
\]
\end{defn}

The Chemin--Lerner spaces are related to the standard spaces
$L^{\theta}_{T}(\dot{B}^{s}_{p,r})$ through Minkowski's inequality. Indeed, it holds that
\[
\|f\|_{\widetilde{L}^{\theta}_{T}(\dot{B}^{s}_{p,r})}
\begin{cases}
\le \|f\|_{L^{\theta}_{T}(\dot{B}^{s}_{p,r})}, & r\ge \theta,\\[1ex]
\ge \|f\|_{L^{\theta}_{T}(\dot{B}^{s}_{p,r})}, & r\le \theta.
\end{cases}
\]

Product estimates in Besov spaces play a fundamental role in the control of nonlinear terms.

\begin{prop}\label{prop3.2}
Let $s>0$ and $1\leq p,r\leq\infty$. Then $\dot{B}^{s}_{p,r}\cap L^{\infty}$ is an algebra and
\[
\|ab\|_{\dot{B}^{s}_{p,r}}
\lesssim
\|a\|_{L^{\infty}}\|b\|_{\dot{B}^{s}_{p,r}}
+
\|b\|_{L^{\infty}}\|a\|_{\dot{B}^{s}_{p,r}}.
\]
Moreover, if $s_{1},s_{2}\leq\frac{d}{p}$ and
$s_{1}+s_{2}>d\max\{0,\frac{2}{p}-1\}$, then
\[
\|ab\|_{\dot{B}^{s_{1}+s_{2}-\frac{d}{p}}_{p,1}}
\lesssim
\|a\|_{\dot{B}^{s_{1}}_{p,1}}
\|b\|_{\dot{B}^{s_{2}}_{p,1}}.
\]
If $s_{1}\leq\frac{d}{p}$, $s_{2}<\frac{d}{p}$ and
$s_{1}+s_{2}\geq d\max\{0,\frac{2}{p}-1\}$, then
\[
\|ab\|_{\dot{B}^{s_{1}+s_{2}-\frac{d}{p}}_{p,\infty}}
\lesssim
\|a\|_{\dot{B}^{s_{1}}_{p,1}}
\|b\|_{\dot{B}^{s_{2}}_{p,\infty}}.
\]
\end{prop}

We also recall composition estimates.

\begin{prop}\label{prop2.25}
Let $F:\mathbb{R}\to\mathbb{R}$ be smooth with $F(0)=0$.
For any $s>0$ and $1\leq p,r\leq\infty$, if
$f\in\dot{B}^{s}_{p,r}\cap L^{\infty}$, then
$F(f)\in\dot{B}^{s}_{p,r}\cap L^{\infty}$ and
\[
\|F(f)\|_{\dot{B}^{s}_{p,r}}\leq C\|f\|_{\dot{B}^{s}_{p,r}},
\]
where $C$ depends on $\|f\|_{L^{\infty}}$, $F$ and its derivatives.

If $s>-d\min(\frac{1}{p},\frac{1}{p'})$ and
$f\in\dot{B}^{s}_{p,r}\cap\dot{B}^{\frac{d}{p}}_{p,1}$, then
$F(f)\in\dot{B}^{s}_{p,r}\cap\dot{B}^{\frac{d}{p}}_{p,1}$ and
\[
\|F(f)\|_{\dot{B}^{s}_{p,r}}\leq C(1+\|f\|_{\dot{B}^{s}_{p,1}})\|f\|_{\dot{B}^{s}_{p,r}},
\]
where $C$ depends on $\|f\|_{L^{\infty}}$, $F$ and its derivatives.
\end{prop}

Finally, we recall the maximal regularity estimates of the heat equation
\[
\left\{
\begin{aligned}
&\partial_t f-\nu\Delta f=g,\\
&f|_{t=0}=f_0.
\end{aligned}
\right.
\]

\begin{lem}\label{lemma6.1}
Let $1\leq p\leq\infty$, $s\in\mathbb{R}$ and $1\leq \rho_1\leq\rho\leq\infty$.
There exists a constant $C$ depending only on $\nu,d,s,r$ such that any solution $f$ satisfies
\[
\|f\|_{\widetilde{L}^{\rho}_{T}(\dot{B}^{s+\frac{2}{\rho}}_{p,r})}
\leq C\left(
\|f_0\|_{\dot{B}^{s}_{p,r}}
+\|g\|_{\widetilde{L}^{\rho_1}_{T}(\dot{B}^{s-2+\frac{2}{\rho_1}}_{p,r})}\right).
\]
\end{lem}

\subsection{Proof of Theorem \ref{thm3.1}: local well-posedness}

We give a brief explanation of the local well-posedness result, namely Theorem~\ref{thm3.1}. 
We consider the two cases $1<p<2d$ and $p=1$ separately.

\begin{itemize}
\item {\emph{Case 1: $1<p<2d$.}}
\end{itemize}

We recall the result proved by Chikami and Ogawa~\cite{CO}: 
if the initial data satisfy
\[
\inf_{x\in\mathbb{R}^d}\rho_0(x)>0,\qquad 
a_{0}\in\dot{B}_{p,1}^{\frac dp-1}\cap\dot{B}_{p,1}^{\frac dp},\qquad 
u_{0}\in\dot{B}_{p,1}^{\frac dp}+\dot{B}_{p,1}^{\frac dp-1},
\]
then there exists a time $T>0$ such that the Cauchy problem \eqref{1.1}--\eqref{1.2} admits a unique solution $(\rho-1,u)$ on $[0,T]\times\mathbb{R}^d$, which satisfies
\begin{equation}\label{r1111}
\begin{aligned}
&\inf_{(t,x)\in [0,T]\times \mathbb{R}^d}\rho(t,x)>0,\qquad 
a\in\mathcal{C}\big([0,T];\dot{B}_{p,1}^{\frac dp-1}\cap\dot{B}_{p,1}^{\frac dp}\big),\\
&u\in\mathcal{C}\big([0,T];\dot{B}_{p,1}^{\frac dp}+\dot{B}_{p,1}^{\frac dp-1}\big)
\cap L^1\big(0,T;\dot{B}_{p,1}^{\frac dp+2}+\dot{B}_{p,1}^{\frac dp+1}\big).
\end{aligned}
\end{equation}

Next, under additionally $a_0\in \dot{B}^{\frac{d}{2}-2}_{2,1}$ and $u_0\in \dot{B}^{\frac{d}{p}-1}_{p,1}$, we establish the regularity properties
\begin{align}\label{addr}
a\in\mathcal{C}([0,T];\dot{B}_{p,1}^{\frac dp-2}),\quad u\in\mathcal{C}([0,T];\dot{B}_{p,1}^{\frac dp-1})\cap L^1(0,T;\dot{B}_{p,1}^{\frac dp+1}).
\end{align}
Note that the high frequencies have already been addressed, and we can establish $L^1_t(\dot{B}^{\frac{d}{p}+1}_{p,1})$-bound for $u$ since
\begin{align}\label{r11111}
\|u\|_{L^1(\dot{B}_{p,1}^{\frac dp+1})}\lesssim T\|u\|_{L^{\infty}(\dot{B}_{p,1}^{\frac dp})}^\ell+\|u\|_{L^1(\dot{B}_{p,1}^{\frac dp+1})}^h.
\end{align}
Indeed, it is sufficient to control the norm $\|(\Lambda^{-1}a,u)\|_{\widetilde{L}^\infty(\dot{B}_{p,1}^{\frac dp-1})}$. Note that
\begin{align}\label{ain}
a=a_0-\int_0^t \dive((1+a)u)\,d\tau.
\end{align}
which gives rise to
\begin{align}\label{a1:1}
\|a\|_{\widetilde{L}^{\infty}_t(\dot{B}^{\frac{d}{p}-2}_{p,1})}\leq \|a_0\|_{\dot{B}^{\frac{d}{p}-2}_{p,1}}+C\int_0^t (1+\|a\|_{\dot{B}^{\frac{d}{p}}_{p,1}})\|u\|_{\dot{B}^{\frac{d}{p}-1}_{p,1}}\,d\tau.
\end{align}
Using the standard maximal regularity estimates for the second equation of \eqref{linearized:1} together with standard product and composite estimates, we have
\begin{equation}\label{u1}
\begin{aligned}
\|u\|_{\widetilde{L}^{\infty}_t(\dot{B}^{\frac{d}{p}-1}_{p,1})}
&\leq C_{\rho} T\|a\|_{L^{\infty}_T(\dot{B}^{\frac{d}{p}}_{p,1})}+C_{\rho}(1+\|a\|_{L^{\infty}_T(\dot{B}^{\frac{d}{p}}_{p,1})})\|a\|_{L^{\infty}_T(\dot{B}^{\frac{d}{p}}_{p,1})}\|u\|_{L^1_t(\dot{B}^{\frac{d}{p}+1}_{p,1})}\\
&\quad+C\int_0^T \Big(\|a\|_{\dot{B}^{\frac{d}{p}-2}_{p,1}}+\|u\|_{\dot{B}^{\frac{d}{p}+1}_{p,1}} \|u\|_{\dot{B}^{\frac{d}{p}-1}_{p,1}}\Big)\,d\tau.
\end{aligned}
\end{equation}
Here $C_{\rho}>0$ is a constant dependent on the upper and lower bounds of $\rho$. By combining \eqref{a1:1}, \eqref{u1} and Gr\"onwall's lemma, we arrive at
\begin{align*}
\|(\Lambda^{-1}a,u)\|_{\widetilde{L}^{\infty}_t(\dot{B}^{\frac{d}{p}-1}_{p,1})}
&\leq e^{C_{\rho}\int_0^T (1+\|a\|_{\dot{B}^{\frac{d}{p}}_{p,1}}+\|u\|_{\dot{B}^{\frac{d}{p}+1}_{p,1}} )\,d\tau}\Big(\|(\Lambda^{-1}a_0,u_0)\|_{\dot{B}^{\frac{d}{p}-1}_{p,1}}\\
&\quad+C_{\rho} T\|a\|_{L^{\infty}_T(\dot{B}^{\frac{d}{p}}_{p,1})}+C_{\rho}(1+\|a\|_{L^{\infty}_T(\dot{B}^{\frac{d}{p}}_{p,1})})\|a\|_{L^{\infty}_T(\dot{B}^{\frac{d}{p}}_{p,1})}\|u\|_{L^1_t(\dot{B}^{\frac{d}{p}+1}_{p,1})}\Big).
\end{align*}
Together with \eqref{r1111}, \eqref{r11111} and standard arguments for time continuity, we have \eqref{addr} and conclude Theorem \ref{thm3.1} when $1<p<2d$.

\begin{itemize}
\item {\emph{Case 2: $p=1$.}}
\end{itemize}

The proof is in the same way as in \cite{BCD,danchin01CPDE} for the Poisson-free case, and we only provide some necessary a-priori estimates for completeness. We assume that there exists a time $T^*=T^*(\rho_0,u_0)$ and a constant $\varepsilon^*$ such that for $t\in [0,T^*]$
\begin{equation}\label{r111e}
\begin{aligned}
&\frac{1}{2}\leq \rho(t,x)\leq 2,\quad \|a\|_{\widetilde{L}^{\infty}_t(\dot{B}^{d}_{1,1})}\leq 2\|a_0\|_{\dot{B}^{d}_{1,1}},\quad \|a\|_{\widetilde{L}^{\infty}_t(\dot{B}^{d-2}_{1,1})}\leq 2\|a_0\|_{\dot{B}^{d-2}_{1,1}},\\
&\|u\|_{\widetilde{L}^{\infty}_t(\dot{B}^{d-1}_{1,1})}\leq 2\|u_0\|_{\dot{B}^{d-1}_{1,1}},\quad \|u\|_{L^1_t(\dot{B}^{d+1}_{1,1})}\leq \varepsilon^*.
\end{aligned}
\end{equation}
We will explain that the above inequalities \eqref{r111e} are, in fact, strict.

Let $\widetilde{U} = u - U$ with $U=e^{\bar {\mathcal{A}}t}u_0$. Since $U$ is only determined by $u_0$, we can find a time $T_1^*=T_1^*(u_0,\varepsilon^*)$ such that $\|U\|_{\widetilde{L}^{\infty}_{T_1^*}(\dot{B}^{d-1}_{1,1})}\leq \|u_0\|_{\dot{B}^{d-1}_{1,1}}$ and $\|U\|_{L^{1}_{T_1^*}(\dot{B}^{d+1}_{1,1})\cap \widetilde{L}^{2}_{T_1^*}(\dot{B}^{d}_{1,1})}\leq (\varepsilon^*)^2$. A standard transport estimate for the first equation of \eqref{linearized:1} leads to
\begin{align}
\|a\|_{\widetilde{L}^{\infty}_t(\dot{B}^{d}_{1,1})}\leq e^{C\|u\|_{L^1_T(\dot{B}^{d-1}_{1,1})}} \|a_0\|_{\dot{B}^{d}_{1,1}}\leq e^{C\varepsilon^*}\|a_0\|_{\dot{B}^{d}_{1,1}}.
\end{align}
Concerning the lower-order estimates, from \eqref{a1} one has
\begin{align*}
\|a\|_{\widetilde{L}^{\infty}_t(\dot{B}^{d-2}_{1,1})}&\leq  \|a_0\|_{\dot{B}^{d-2}_{1,1}}+C\int_0^t (\|u\|_{\dot{B}^{d-1}_{1,1}}+\|u\|_{\dot{B}^{d-1}_{1,1}} \|a\|_{\dot{B}^{d}_{1,1}})\,d\tau\\
&\leq  \|a_0\|_{\dot{B}^{d-2}_{1,1}}+C(1+\|a_0\|_{\dot{B}^{d}_{1,1}})\|u_0\|_{\dot{B}^{d-1}_{1,1}}.
\end{align*}
By \eqref{linearized:1}, $\widetilde{U}$ has to satisfy
\begin{equation*}
\partial_{t}\widetilde{U}-\bar{\mathcal{A}}\widetilde{U}=-\nabla a-\frac{\nabla}{-\Delta} a+g,\quad \widetilde{U}(0,x)=0.
\end{equation*}
Employing Lemma \ref{lemma6.1} to the above equation yields
\begin{align*}
\|\widetilde{U}\|_{\widetilde{L}^{\infty}_t(\dot{B}^{d-1}_{1,1}) \cap L^1_t(\dot{B}^{d+1}_{1,1})}&\leq C T\|a\|_{\widetilde{L}^{\infty}_t(\dot{B}^{d-2}_{1,1}\cap \dot{B}^{d}_{1,1})} +C\|g\|_{L^1_T(\dot{B}^{d-1}_{1,1})}\\
&\leq C T\|a_0\|_{\dot{B}^{d-2}_{1,1}\cap \dot{B}^{d}_{1,1}}+C(\varepsilon^*)^4+C(\varepsilon^*)^2\|u_0\|_{\dot{B}^{d-1}_{1,1}}+C(\varepsilon^*)^2\\
&\quad+(1+\|a_0\|_{\dot{B}^{d-1}_{1,1}})\|a_0\|_{\dot{B}^{d-1}_{1,1}} \|\widetilde{U}\|_{L^1_t(\dot{B}^{d+1}_{1,1})}
\end{align*}
Here, we used
\begin{align*}
\|g_1\|_{L^1_T(\dot{B}^{d-1}_{1,1})}&\lesssim \|U\|_{\widetilde{L}^2_t(\dot{B}^{d}_{1,1})}^2+\|U\|_{\widetilde{L}^{\infty}_t(\dot{B}^{d-1}_{1,1})}\|\widetilde{U}\|_{L^1_t(\dot{B}^{d+1}_{1,1})}+\|\widetilde{U}\|_{\widetilde{L}^2_t(\dot{B}^{d}_{1,1})}^2,\\
\|g_2\|_{L^1_T(\dot{B}^{d-1}_{1,1})}&\lesssim (1+\|a\|_{\widetilde{L}^{\infty}_t(\dot{B}^{d}_{1,1})})\|a\|_{\widetilde{L}^{\infty}_t(\dot{B}^{d}_{1,1})} (\|U\|_{L^1_t(\dot{B}^{d+1}_{1,1})}+\|\widetilde{U}\|_{L^1_t(\dot{B}^{d+1}_{1,1})}),\\
\|g_3\|_{L^1_T(\dot{B}^{d-1}_{1,1})}&\leq CT \|a\|_{\widetilde{L}^{\infty}_t(\dot{B}^{d}_{1,1})}.
\end{align*}
Combining the above estimates and $\|a_0\|_{\dot{B}^{d}_{1,1}}\ll1$, one can first choose a suitably small $\varepsilon_1^*$ and then choose a small time $T^*\leq T_1^*$ such that \eqref{r111e} holds true. With the above a-priori estimates and standard iteration process, one can complete the existence part.

For uniqueness, we assume that $(\rho_1,u_1)$ and $(\rho_2,u_2)$ are two solutions on $[0,T]\times \mathbb{R}^d$ satisfying the regularity properties in Theorem \ref{thm3.1} with the same initial data. Then, $(\widetilde{\rho},\widetilde{u})=(\rho_1-\rho_2,u_1-u_2)$ satisfies
\begin{align*}
&\partial_t \widetilde{\rho}+u_1\cdot\nabla \widetilde{\rho}+\widetilde{\rho} \dive u_2=-\widetilde{u}\cdot \nabla \rho_2-(1+\rho_1)\dive \widetilde{u},\\
&\partial_t \widetilde{u}-\bar{\mathcal{A}}\widetilde{u}=-\nabla \widetilde{\rho}-\frac{\nabla}{-\Delta}\widetilde{\rho}+g(\rho_1,u_1)-g(\rho_2,u_2).
\end{align*}
In the case $d\geq 3$, one may carry out computations using standard product laws to obtain
\begin{align*}
\|\widetilde{\rho}\|_{\widetilde{L}^{\infty}_t(\dot{B}^{d-1}_{1,1})}&\lesssim \int_0^t \|\widetilde{u}\|_{\dot{B}^{d}_{1,1}}\,d\tau,\\
\|\widetilde{\rho}\|_{\widetilde{L}^{\infty}_t(\dot{B}^{d-3}_{1,1})}&\lesssim \int_0^t\|\dive ( u_1 \widetilde{\rho}+(1+\rho_1)\widetilde{u})\|_{\dot{B}^{d-3}_{1,1}}\,d\tau\lesssim \int_0^t (\|\widetilde{\rho}\|_{\dot{B}^{d-1}_{1,1}}+ \|\widetilde{u}\|_{\dot{B}^{d-2}_{1,1}})\,d\tau,
\end{align*}
and
\begin{align*}
\| \widetilde{u}\|_{\widetilde{L}^{\infty}_t(\dot{B}^{d-1}_{1,1})\cap L^1_t(\dot{B}^{d+1}_{1,1})}&\lesssim \int_0^t(1+\|u_1\|_{\dot{B}^{d+1}_{1,1}}) \|\widetilde{\rho}\|_{\dot{B}^{d-2}_{1,1}\cap \dot{B}^{d-1}_{1,1}}\,d\tau+\int_0^t \|\rho_2-1\|_{\dot{B}^{d}_{1,1}} \|\widetilde{u}\|_{\dot{B}^{d+1}_{1,1}}\,d\tau.
\end{align*}
Taking Gr\'onwall's lemma and the fact that $\|\rho_2-1\|_{L^{\infty}_t(\dot{B}^{d}_{1,1})}$ is suitably small, we conclude $(\delta \rho, \delta u)=0$ a.e. on $[0,T]\times \mathbb{R}^d$. The case $d=2$ can be treated by replacing the above $\dot{B}^{d-1}_{1,1}$-estimates with the weaker $\dot{B}^{d-1}_{1,\infty}$-estimates and applying a log-type inequality; see \cite{BCD}. The details are omitted.

\bigbreak

\bigbreak\noindent
\noindent \textbf{Acknowledgments.~} L.-Y. Shou is supported by National Natural Science Foundation of China (Grant No. 12301275).

\bigbreak

\noindent
\textbf{Conflict of interest.} The authors do not have any possible conflicts of interest.

\bigbreak

\noindent
\textbf{Data availability statement.}
 Data sharing is not applicable to this article, as no data sets were generated or analyzed during the current study.

\bigbreak
\bigbreak

\bigbreak

\noindent (L.-Y. Shou)\par\nopagebreak
\noindent\textsc{School of Mathematical Sciences, Ministry of Education Key Laboratory of NSLSCS, and Key Laboratory of Jiangsu Provincial Universities of FDMTA, Nanjing Normal University, Nanjing, 210023, P. R. China}

Email address: {\texttt{shoulingyun11@gmail.com}}

\bigbreak
\bigbreak

\noindent (Z. Song)\par\nopagebreak
\noindent\textsc{Mathematics and Key Laboratory of Mathematical MIIT, Nanjing University of Aeronautics and
Astronautics, Nanjing, 211106, P. R. China}

Email address: {\texttt{szh1995@nuaa.edu.cn; songzh19950504@gmail.com}}

\end{document}